\documentclass[11pt]{article}
\usepackage{PRIMEarxiv}

\usepackage{amsmath}
\usepackage{amssymb}
\usepackage{amsfonts}
\usepackage{amsthm}
\usepackage{standalone}
\usepackage{float}
\usepackage{graphicx}
\usepackage{epstopdf}
\usepackage{hyperref}
\usepackage{xr}
\usepackage{algorithm}
\usepackage{algpseudocode}
\usepackage{cleveref}
\usepackage{tabularx}
\usepackage{booktabs}
\usepackage{tikz}
\usetikzlibrary{shapes.geometric, arrows.meta, positioning, calc}

\tikzset{
    startend/.style = {rectangle, rounded corners, minimum width=3cm, minimum height=1cm, text centered, draw=black, fill=green!15},
    process/.style = {rectangle, rounded corners, minimum width=3cm, minimum height=1cm, text centered, text width=4.5cm, draw=black, fill=blue!15},
    decision/.style = {rectangle, rounded corners, minimum width=3cm, minimum height=1cm, text centered, text width=4cm, draw=black, fill=gray!10},
    io/.style = {rectangle, rounded corners, minimum width=3cm, minimum height=1cm, text centered, text width=4cm, draw=black, fill=red!15},
    arrow/.style = {thick, ->, >=Stealth}
}

\ifpdf
  \DeclareGraphicsExtensions{.eps,.pdf,.png,.jpg}
\else
  \DeclareGraphicsExtensions{.eps}
\fi

\newtheorem{theorem}{Theorem}

\newtheorem{assumption}{Assumption}
\newtheorem{remark}{Remark}

\theoremstyle{definition}

\crefname{hypothesis}{Hypothesis}{Hypotheses}
\crefname{fact}{Fact}{Facts}
\crefname{assumption}{Assumption}{Assumptions}
\Crefname{assumption}{Assumption}{Assumptions}
\crefname{algorithm}{algorithm}{algorithms}
\Crefname{algorithm}{Algorithm}{Algorithms}

\usepackage{amsopn}

\ifpdf
\hypersetup{
  pdftitle={SVD-Preconditioned Gradient Descent Method for Solving Nonlinear Least Squares Problems},
  pdfauthor={Zhipeng Chang, Wenrui Hao, Nian Liu}
}
\fi

\title{SVD-Preconditioned Gradient Descent Method for Solving Nonlinear Least Squares Problems}

\author{
  Zhipeng Chang \\
  Department of Mathematics \\
  Penn State University \\
  University Park, PA 16802, USA \\
  \texttt{zfc5231@psu.edu}
  \And
  Wenrui Hao$^*$ \\
  Department of Mathematics \\
  Penn State University \\
  University Park, PA 16802, USA \\
  \texttt{wxh64@psu.edu}
  \And
  Nian Liu \\
  Department of Mathematics \\
  Penn State University \\
  University Park, PA 16802, USA \\
  \texttt{nkl5330@psu.edu}
}

\fancypagestyle{firstpage}{%
  \fancyhf{}

  \lfoot{\footnotesize\textit{Submitted to Foundations of Computational Mathematics, \today.}\\$^*$Corresponding author.}
  \rfoot{\footnotesize\thepage}
}

\begin{document}

\maketitle
\thispagestyle{firstpage}

\begin{abstract}
This paper introduces a novel optimization algorithm designed for nonlinear least-squares problems. The method is derived by preconditioning the gradient descent direction using the Singular Value Decomposition (SVD) of the Jacobian. This SVD-based preconditioner is then integrated with the first- and second-moment adaptive learning rate mechanism of the Adam optimizer. We establish the local linear convergence of the proposed method under standard regularity assumptions and prove global convergence for a modified version of the algorithm under suitable conditions. The effectiveness of the approach is demonstrated experimentally across a range of tasks, including function approximation, partial differential equation (PDE) solving, and image classification on the CIFAR-10 dataset. Results show that the proposed method consistently outperforms standard Adam, achieving faster convergence and lower error in both regression and classification settings.

\end{abstract}



\section{Introduction}
Optimization is a central object of study in computational mathematics, underpinning the analysis and numerical solution of a wide range of problems in scientific computing. Many numerical algorithms, including those arising from discretizations of partial differential equations (PDEs), inverse problems, and data-driven models, can be naturally formulated as optimization problems. Classical methods such as gradient descent~\cite{ruder2016overview,hao2021gradient}, Newtonâ€™s method~\cite{hao2024gauss}, and quasi-Newton methods~\cite{nocedal-2006-numerical_optimization,hao2025efficient} (e.g., L-BFGS~\cite{liu-1989-LBFGS}) form the theoretical backbone of numerical optimization and have been studied extensively from analytical and algorithmic perspectives.

In recent years, renewed attention has been given to Newton-type methods and gradient flow formulations in contexts motivated by scientific machine learning, inverse problems, and variational models. Randomized Newton and Gauss--Newton methods have been proposed to reduce computational complexity while retaining essential curvature information~\cite{hao2024gauss,chen2022randomized}. Related work has explored energy-based formulations and greedy strategies to improve stability and approximation properties of neural network discretizations for PDEs~\cite{siegel2021greedy}. More broadly, these developments reflect an ongoing effort to design optimization algorithms that respect the analytical structure of the underlying problem rather than relying solely on heuristic modifications.

From a foundational perspective, second-order methods, including Newtonâ€™s method, quasi-Newton schemes~\cite{nocedal-2006-numerical_optimization}, and natural gradient descent~\cite{amari-1998-natural_gradient,muller2023achieving}, achieve rapid local convergence by incorporating curvature and geometric information of the objective functional. However, the explicit construction and storage of second-order operators are often computationally prohibitive, particularly for high-dimensional problems or in the presence of ill-conditioned Jacobians. Moreover, for nonconvex objectives, these methods may exhibit sensitivity to initialization and lack robust global behavior.

By contrast, first-order methods are attractive due to their simplicity and low per-iteration cost, but they typically converge slowly in the presence of strong anisotropy or poor conditioning \cite{park2020additive}. This trade-off has motivated the development of approximate second-order techniques that seek to capture essential spectral information while remaining computationally feasible. Adaptive gradient methods, such as 
Additive Schwarz methods \cite{park2020additive,park2021accelerated},
AdaGrad~\cite{duchi-2011-adagrad}, RMSProp, and Adam~\cite{kingma-2017-adam}, introduce diagonal or low-complexity preconditioning based on gradient statistics. More structured approaches, including K-FAC~\cite{martens-2015-KFAC} and Shampoo~\cite{gupta-2018-shampoo}, exploit tensor and Kronecker product structures. While effective in certain regimes, these methods are typically motivated by empirical considerations and provide limited insight into the interaction between algorithmic design and the analytical structure of the underlying optimization problem.

Motivated by these considerations, we propose a new optimization framework for least-squares problems grounded in a dynamical systems and operator-theoretic perspective. Rather than modifying classical descent directions heuristically, the proposed method is derived by coupling gradient flow dynamics with a preconditioned operator constructed from a singular value decomposition (SVD) of the local Jacobian. This construction yields a descent direction that selectively incorporates dominant singular modes, providing a principled mechanism for addressing ill-conditioning and anisotropy inherent in least-squares formulations. In addition, adaptive step-size and momentum mechanisms are introduced within this framework to enhance robustness while preserving the underlying continuous-time interpretation.

The proposed approach provides a mathematically transparent link between first- and second-order optimization dynamics by incorporating problem-dependent spectral information through an operator-based preconditioner. Although the method is motivated by least-squares problems, the underlying framework is not restricted to this setting and can be naturally extended to more general nonlinear optimization and variational formulations, including cross-entropy classifications. In particular, the dynamical systems viewpoint and the use of local spectral information offer a unifying perspective for a broad class of optimization problems arising in inverse problems and scientific computing.

The remainder of this paper is organized as follows. In~\cref{sec:numerical method}, we derive the proposed method from a continuous-time perspective and present its discrete numerical formulation, together with an analysis of local convergence. In~\cref{sec:adam framework with spgd}, we extend the framework to an Adam-type setting and analyze the global convergence of the resulting coupled algorithm. In~\cref{subsec:function_approximation,subsec:poisson_equation,subsec:public_dataset}, we demonstrate the efficiency of the proposed method through several representative numerical examples.

\section{SVD-Preconditioned Gradient Descent (SPGD) method}
\label{sec:numerical method}
In this section, we consider the following nonlinear least squares problem
\begin{equation}\label{equ:objective function}
\min_{\theta\in\mathbb{R}^m} f(\theta) := \frac{1}{2}\|F(\theta)\|_2^2,
\end{equation}
where $\theta\in \mathbb{R}^m$ and $ F: \mathbb{R}^m\rightarrow \mathbb{R}^n$. A common approach to solving this problem is to follow the continuous gradient flow, which is given by
\begin{equation}\label{equ:classical gradient flow}
    \frac{\mathrm{d}\theta}{\mathrm{d}t}=-J_F^{\top}(\theta)F(\theta),
\end{equation}
where $J_F (\theta)$ denotes the Jacobian matrix of $F$ with respect to $\theta$. The gradient flow ensures the monotonic decrease of $\|F(\theta)\|_2^2$ along the trajectory $t\mapsto \theta(t)$. However, its convergence tends to be slow near a stationary point $\theta^{*}$ satisfying $J_F^{\top}(\theta^*)F(\theta^*)$=0, since the linearized dynamics around $\theta^*$ might be  governed by the ill-conditioned Fisher Information matrix $J_F^{\top}(\theta^*)J_F(\theta^*).$ As established in \cite{karakida-2019-universal_statistics_of_FIM}, such matrices in deep networks inherently exhibit a pathological eigenvalue spectrum, resulting in severe ill-conditioning.

\subsection{Method derivation}
In many scientific and engineering applications, nonlinear least-squares formulations arise naturally from the discretization of governing equations. As a representative example, consider a discretized nonlinear partial differential equation (PDE) of the form
$F(\theta) = \Delta_h \theta + g(\theta),$
where $\theta$ denotes a numerical approximation to the solution of the continuous problem $-\Delta u = g(u)$, and $\Delta_h$ is a discrete Laplacian operator obtained, for example, via finite difference or finite element methods~\cite{grossmann-2007-numerical_treatment}.  
In this setting, one can also consider the corresponding \emph{physical flow}, which takes the form of a parabolic evolution equation:
\begin{equation}\label{equ:physical flow}
    \frac{\mathrm{d}\theta}{\mathrm{d}t} = F(\theta), \quad \theta(0) = \theta_0.
\end{equation}
Here, $F$ may represent a more general operator incorporating transport, diffusion, or reaction mechanisms. The steady state $\theta^*$ of~\cref{equ:physical flow} satisfies $F(\theta^*) = 0$, and thus corresponds to a minimizer of the nonlinear least-squares problem~\cref{equ:objective function}.  
In practice, such steady states can often be obtained by numerically integrating the discretized physical flow~\cref{equ:physical flow} until it reaches its long-time limit, provided the solution converges. This approach is frequently more efficient than the residual-based gradient flow~\cref{equ:classical gradient flow}, as it evolves according to the intrinsic physical dynamics encoded in $F$, rather than enforcing a purely monotonic decay of $\|F(\theta)\|$. 

A natural question is how the convergence of gradient flow for least-squares problems can be accelerated by incorporating information from the underlying physical or operator dynamics. Motivated by the observation that physical evolution equations often exhibit faster relaxation behavior, we introduce a constant preconditioner $A$ to modify the standard gradient flow. Specifically, we consider the preconditioned least-squares problem
\begin{equation}
    \min_{\theta \in \mathbb{R}^m} \frac{1}{2}\|A F(\theta)\|_2^2,
\end{equation}
which induces the following preconditioned gradient flow:
\begin{equation*}
    \frac{\mathrm{d}\theta}{\mathrm{d}t}
    = -\frac{1}{2}\nabla \|A F(\theta)\|_2^2
    = -J_F^{\top}(\theta)\, A^{\top} A\, F(\theta),
\end{equation*}
where $J_F(\theta)$ denotes the Jacobian of $F$ evaluated at $\theta$.

Let the singular value decomposition (SVD) of $J_F(\theta)$ be given by $
J_F(\theta) = U \Sigma V^{\top}. $
Choosing the preconditioner as
$
A = V \Sigma^{-1/2} U^{\top},
$
we obtain
\begin{equation}\label{equ:preconditioned gradient flow}
\begin{aligned}
    \frac{\mathrm{d}\theta}{\mathrm{d}t}
    &= -J_F^{\top}(\theta) A^{\top} A F(\theta) 
    = -(V \Sigma U^{\top})(U \Sigma^{-1/2} V^{\top})(V \Sigma^{-1/2} U^{\top}) F(\theta) \\
    &= -V U^{\top} F(\theta).
\end{aligned}
\end{equation}
The resulting operator $V U^{\top}$ is semi-orthogonal, and the induced dynamics can be interpreted as a flow that preserves the stability properties of gradient descent while effectively filtering the residual through the dominant singular directions of the Jacobian. In this sense, the preconditioning aligns the descent direction with the intrinsic operator structure of the problem, leading to accelerated convergence compared to the classical gradient flow.

To obtain a practical iterative algorithm, we discretize the preconditioned gradient flow~\eqref{equ:preconditioned gradient flow}. This yields the update rule
\begin{equation}\label{equ:update for theta at iteration t}
    \theta_{t+1} = \theta_t - \alpha \, V_t U_t^{\top} F(\theta_t),
\end{equation}
where $\alpha > 0$ denotes the step size and $J_F(\theta_t) = U_t \Sigma_t V_t^{\top}$
is the (economy-sized) singular value decomposition of the Jacobian evaluated at the current iterate $\theta_t$. The resulting scheme defines a descent direction that is adaptively aligned with the dominant local singular structure of the nonlinear operator $F$.

\subsection{Local convergence}
We establish the {local convergence} of both the {GD} and {SPGD} algorithms for (\ref{equ:objective function}) under the following assumptions.

\begin{assumption}\label{assump:regular}
$\quad$

\begin{itemize}
    \item The mapping $F$ admits a zero equilibrium $\theta^* \in \mathbb{R}^m$; that is $F(\theta^*) = 0$.
    \item There exists a neighborhood $U$ of $\theta^*$ such that for all $\theta \in U$, the Jacobian matrix $J_F(\theta)$ satisfies $\mathrm{rank}(J_F(\theta)) \equiv r$. Thus, by constant-rank level set theorem~\cite{Lee-2012}, we have that 
$
\mathcal{M} := \{\theta \in U : F(\theta)=0\}
$
is an embedded submanifold of $\mathbb{R}^m$ with codimension $r$. 
    \item The function $f(\theta)$ is $L$-smooth on $U$; specifically,
    \begin{equation*}
        \|\nabla f(\theta_1) - \nabla f(\theta_2)\|_2 \leq L \|\theta_1 - \theta_2\|_2, \qquad \forall\, \theta_1, \theta_2 \in U.
    \end{equation*}
\end{itemize}
\end{assumption}

\begin{theorem}[Local linear convergence of GD near a regular equilibrium]\label{thm:local linear convergence of gd}
Suppose that $F$ satisfies the conditions in Assumption~\ref{assump:regular}. 
Consider the GD iteration
\begin{equation}\label{equ:gd iteration in thm}
    \theta_{t+1} = \theta_t - \alpha\nabla f(\theta_t),
    \qquad\forall t\ge 0,
\end{equation}
where $\alpha>0$ denotes the step size. Then, for any sufficiently small $\alpha$, there exists a neighborhood
$U_{\alpha}\subset U$ of $\theta^*$ such that the gradient descent sequence
$\{\theta_t\}_{t\geq 0}$ generated by~\cref{equ:gd iteration in thm} with
initial point $\theta_0\in U_{\alpha}$ satisfies
\begin{equation}\label{equ:linear convergence of gd}
f(\theta_t)\leq(1-\alpha\mu)^tf(\theta_0),
\qquad \forall t\geq 0,
\end{equation}
where $\displaystyle\mu=\max_{\theta\in V}  \frac{ \|\nabla f(\theta)\|_2^2} {2 f(\theta)}$ as defined in~Lemma~2.
In particular, the sequence $\{f(\theta_t)\}_{t\geq 0}$ converges to zero at a linear rate.
\end{theorem}

\begin{proof}
Let $V$ be the neighborhood of $\theta^*$ defined in~Lemma~1. For a given step size $0<\alpha \leq 1/L$, we define
\begin{equation}\label{equ:definition of Ualpha}
    U_{\alpha} = \left\{\theta \in V : \mathrm{dist}(\theta, \partial V) > \frac{2}{\mu} \sqrt{\frac{2 f(\theta)}{\alpha}} \right\}.
\end{equation}
Since the functions are continuous, $U_{\alpha} \subset V$ forms an open neighborhood of $\theta^*$.

We proceed by induction to show that any gradient descent sequence
$\{\theta_t\}_{t \ge 0}$ initialized at $\theta_0 \in U_{\alpha}$ remains in $V$ for all $t \ge 0$.
This ensures that all the properties established in
Lemmas~1 and~2
hold along the entire trajectory. Moreover, the corresponding sequence $\{f(\theta_t)\}_{t \ge 0}$
converges to zero at a linear rate.

\medskip
\noindent
\textbf{Base case.} Since $\theta_0\in U_{\alpha}\subset V$, the claim of~\cref{thm:local linear convergence of gd} holds trivially for $t=0$. 

\medskip
\noindent
\textbf{Induction hypothesis.} Assume that, for some $t \ge 1$, the iterate $\theta_t \in V$ satisfies
\begin{align}
    f(\theta_t) &\le (1 - \alpha \mu)^t f(\theta_0), \label{equ:function induction for thetat} \\
    \|\theta_t - \theta_0\|_2
    &\le \sum_{k=0}^{t-1} \sqrt{2\alpha f(\theta_0)} (1 - \alpha \mu)^{k/2}
    < \frac{2}{\mu} \sqrt{\frac{2 f(\theta_t)}{\alpha}}.
    \label{equ:distance_induction_gd}
\end{align}
It remains to prove that these properties hold for $\theta_{t+1}$ as well.

\medskip
\noindent
\textbf{Induction step.} By the descent lemma, for any $\theta\in U$, we have
\begin{equation*}
    f(\theta-\alpha\nabla f(\theta))\leq f(\theta)-\alpha\|\nabla f(\theta)\|_2^2+\frac{L\alpha^2}{2}\|\nabla f(\theta)\|_2^2.
\end{equation*}
For any sufficiently small step size $0<\alpha\leq1/L$, this inequality simplifies to
$    f(\theta-\alpha\nabla f(\theta))\leq f(\theta)-\frac{\alpha}{2}\|\nabla f(\theta)\|_2^2.$
By setting $\theta=\theta_t$, we have 
\begin{equation}\label{equ:descent at thetat}
    f(\theta_{t+1})\leq f(\theta_t)-\frac{\alpha}{2}\|\nabla f(\theta_t)\|_2^2.
\end{equation}
Moreover, from the update rule in~\cref{equ:gd iteration in thm}, we have
\begin{equation}\label{equ:distance of adjacent step}
    \|\theta_{t+1}-\theta_t\|_2=\alpha\|\nabla f(\theta_t)\|\leq\sqrt{2\alpha f(\theta_t)}.
\end{equation}
Combining~\cref{equ:function induction for thetat} and~\cref{equ:distance of adjacent step} gives
$    \|\theta_{t+1}-\theta_t\|_2\leq\sqrt{2\alpha f(\theta_0)}\left(1-\alpha\mu\right)^{\frac{t}{2}}.
$
Since $\theta_t\in V$, by applying~Lemma~2, we have
\begin{equation}\label{equ:PL inequality at thetat}
    \|\nabla f(\theta_t)\|_2^2\geq 2\mu f(\theta_t).
\end{equation}
Substituting~\cref{equ:descent at thetat} into~\cref{equ:PL inequality at thetat} gives
$    f(\theta_{t+1})\leq (1-\alpha\mu)f(\theta_t).$
Using the induction hypothesis~\cref{equ:function induction for thetat} again, we obtain
$    f(\theta_{t+1})\leq (1-\alpha\mu)^{t+1}f(\theta_0).
$
Finally, summing the stepwise distances yields 
\begin{equation}
\begin{split}
    \|\theta_{t+1}-\theta_0\|_2&\leq\sum_{k=0}^t\|\theta_{k+1}-\theta_k\|
    \leq\sum_{k=0}^t \sqrt{2\alpha f(\theta_0)}\left(1-\alpha\mu\right)^{\frac{k}{2}}
    <\frac{\sqrt{2\alpha f(\theta_0)}}{1-\sqrt{1-\alpha\mu}}
    \\
    &\leq\frac{2}{\mu}\sqrt{\frac{2f(\theta_0)}{\alpha}}.
  \end{split}
\end{equation}
By the construction~\cref{equ:definition of Ualpha} of $U_{\alpha},$ this implies $ \|\theta_{t+1}-\theta_0\|_2<\mathrm{dist}(\theta_0,\partial V)$, and therefore $\theta_{t+1}\in V$. Hence, the induction step holds.
By induction, the GD sequence $\{\theta_t\}_{t\geq 0}$ generated by~\cref{equ:gd iteration in thm} with $\theta_0\in U_{\alpha}$ remains in $V$ for all $t\geq 0$, and satisfies
$    f(\theta_{t})\leq(1-\alpha\mu)^t f(\theta_0),$
which completes our proof.
\end{proof}

\begin{remark}\label{rem: convergence rate for gd}
Locally in a neighborhood of $\theta^*$, the constants $\mu$ and $L$ admit a natural spectral interpretation. 
The Lipschitz constant $L$ is of the same order as $\sigma_{\max}^2(J_F(\theta^*))$, while the constant $\mu$ appearing in~Lemma~2 is of the same order as $\sigma_{\min}^2(J_F(\theta^*))$. 
Choosing the step size $\alpha = 1/L$, the corresponding linear convergence factor satisfies
\begin{equation*}
    1 - \frac{\mu}{L}
    \;\approx\;
    1 - \frac{\sigma_{\min}^2(J_F(\theta^*))}{\sigma_{\max}^2(J_F(\theta^*))},
\end{equation*}
which recovers the classical local convergence behavior of GD and shows that the rate is governed by the square of the condition number of $J_F(\theta^*)$.
\end{remark}

To establish the local convergence of the SPGD method, we first introduce two key lemmas  in {\bf Appendix}. Specifically,~Lemma~3 provides a lower bound control of $\|F(\theta)\|_2$ via the distance $n(\theta)=\theta-\pi(\theta)$ from the $\theta$ to the manifold $\mathcal{M}$, with the bound quantified by the smallest singular value of the Jacobian.~Lemma~4 shows that the projection of $F(\theta)$ onto the orthogonal complement of the Jacobianâ€™s column space is a higher-order small term of $\|F(\theta)\|$, which can be neglected in the convergence analysis.

\begin{theorem}[Local linear convergence of the SPGD method near a regular equilibrium]\label{thm:local linear convergence of pgd}
Suppose that $F$ satisfies the conditions in Assumption~\ref{assump:regular}. Let $\pi \colon W \to \mathcal{M}$ denote the projection map defined in~Lemma~4, and let $\sigma_{\min}$ and $\sigma_{\max}$ be positive constants such that the singular values of $J_F(\theta)$ satisfy
$\sigma_{\min} \le \sigma(J_F(\theta)) \le \sigma_{\max}, \forall \theta \in W.$
Then, for any sufficiently small step size $\alpha \le 1/\sigma_{\max}$, there exists a neighborhood $U_{\alpha} \subset W$ of $\theta^*$ such that, for any initial point $\theta_0 \in U_{\alpha}$, the sequence $\{\theta_t\}_{t \ge 0}$ generated by~\eqref{equ:update for theta at iteration t} satisfies
\begin{equation}\label{equ:linear convergence of pgd}
    f(\theta_t) \le \left(1 - \frac{\alpha \sigma_{\min}}{2}\right)^t f(\theta_0),
    \qquad \forall t \ge 0.
\end{equation}
In particular, the objective values $\{f(\theta_t)\}_{t \ge 0}$ converge to zero at a linear rate.
\end{theorem}

\begin{proof}
According to the construction of $W$ in~Lemma~4, by shrinking $W$ if necessary, we assume that $W$ is precompact, and for any $\theta\in W$,
\begin{equation}\label{equ:shrinking W}
    \left(C_{\perp} + \frac{L_{J_F}\alpha^2}{2}\right) \|F(\theta)\|_2 \le \frac{\alpha\sigma_{\min}}{2},
\end{equation}
where $L_{J_F}$ is the Lipschitz constant of $J_F$ on $W$, that is,
\begin{equation}\label{equ:Lipschitz constant for jacobian}
    \|J_F(\theta_1)-J_F(\theta_2)\|_2\leq L_{J_F}\|\theta_1-\theta_2\|_2.
\end{equation}
Notice that $L_{J_F}$ is finite due to the precompactness of $W$ and the $C^2$ smoothness of $F$.
We define the neighborhood $U_{\alpha}\subset W$ of $\theta^*$ as:
\begin{equation}\label{equ:definition of Ualpha_pgd}
    U_{\alpha} = \left\{\theta \in W : \mathrm{dist}(\theta, \partial W) > \frac{2}{\sigma_{\min}} \|F(\theta)\|_2 \right\},
\end{equation}
where $C_{\perp}$ is defined in~Lemma~4.
Since the functions appearing in~\cref{equ:definition of Ualpha_pgd} are continuous, $U_{\alpha}\subset W$ is an open neighborhood of $\theta^*$.

We proceed by induction to show that any sequence $\{\theta_t\}_{t \ge 0}$ generated by the update rule $\theta_{t+1} = \theta_t - \alpha V_t U_t^\top F(\theta_t)$ initialized at $\theta_0 \in U_{\alpha}$ remains in $W$ for all $t \ge 0$, and the residual norm converges linearly.

\medskip
\noindent
\textbf{Base case.} Since $\theta_0 \in U_{\alpha} \subset W$, the claim holds trivially for $t=0$.

\medskip
\noindent
\textbf{Induction hypothesis.} Assume that, for some $t \ge 0$, the iterate $\theta_t \in W$ satisfies
\begin{align}
    \|F(\theta_t)\|_2 &\le \left(1 - \frac{\alpha\sigma_{\min}}{2}\right)^t \|F(\theta_0)\|_2, \label{equ:residual induction for thetat} \\
    \|\theta_t - \theta_0\|_2 &\le \sum_{k=0}^{t-1} \alpha \left(1 - \frac{\alpha \sigma_{\min}}{2}\right)^k \|F(\theta_0)\|_2
    < \frac{2}{\sigma_{\min}} \|F(\theta_0)\|_2. \label{equ:distance_induction_pgd}
\end{align}
It remains to prove that these properties hold for $\theta_{t+1}$ as well.

\medskip
\noindent
\textbf{Induction step.} 
We apply the fundamental theorem of calculus to $F$ and obtain
\begin{equation}
\begin{split} F(\theta_{t+1})&=F(\theta_t)+\int_0^1 J_F(\theta_t+s(\theta_{t+1}-\theta_t))(\theta_{t+1}-\theta_t)\mathrm{d}s
    \\
    & =F(\theta_t)+\int_0^1 J_F(\theta_t)(\theta_{t+1}-\theta_t)\mathrm{d}s+\int_0^1 [J_F(\theta_t+s(\theta_{t+1}-\theta_t))-J_F(\theta_t)](\theta_{t+1}-\theta_t)\mathrm{d}s        \\
    &=F(\theta_t)+J_F(\theta_t)(\theta_{t+1}-\theta_t)+Q_t.
\end{split}  \label{equ:relationship of F at t and t+1}
\end{equation}
Due to $\theta_{t+1}-\theta_t=-\alpha V_tU_t^{\top}F(\theta_t)$ and $J_F(\theta_t) = U_t \Sigma_t V_t^\top$, we obtain
\begin{equation}\label{equ:iteration scheme for F}
    F(\theta_{t+1})=M_tF(\theta_t)+Q_t \hbox{~and~}M_t=I-\alpha U_t\Sigma_t U_t^{\top}.
\end{equation}
Using~\cref{equ:Lipschitz constant for jacobian} yields
\begin{equation}\label{equ:estimation for Qt}
\begin{split}
    \|Q_t\|_2&\leq \int_0^1\|J_F(\theta_t+s(\theta_{t+1}-\theta_t))-J_F(\theta_t)\|_2\|(\theta_{t+1}-\theta_t)\|_2\mathrm{d}s
    \\
    &\leq \int_0^1 sL_{J_F}\|\theta_{t+1}-\theta_t\|_2^2\mathrm{d}s
    =\frac{L_{J_F}}{2}\|\theta_{t+1}-\theta_t\|_2^2\leq \frac{L_{J_F}\alpha^2}{2}\|F(\theta_t)\|_2^2.
\end{split}
\end{equation}

Now, let $\mathcal{R}_t = \mathrm{range}(J_F(\theta_t))$ denote the column space of the Jacobian matrix $J_F(\theta_t)$, and let $\mathcal{R}_t^\perp$ denote its orthogonal complement. We define the orthogonal projection operators $P_{\mathcal{R}_t}$ and $P_{\mathcal{R}_t^\perp}$ respectively and have the orthogonal decomposition of the residual $F(\theta_t)$: 
\begin{equation}\label{equ:orthogonal projection of residual}
    F(\theta_t) = P_{\mathcal{R}_t} F(\theta_t) + P_{\mathcal{R}_t^\perp} F(\theta_t).
\end{equation}
Therefore, combining~\cref{equ:iteration scheme for F} and~\cref{equ:orthogonal projection of residual}, we obtain
\begin{equation}\label{equ:final relationship of F at t and t+1}
    F(\theta_{t+1})= M_tP_{\mathcal{R}_t} F(\theta_t) + M_tP_{\mathcal{R}_t^\perp} F(\theta_t)+ Q_t.
\end{equation}

The operator $M_t$ acts differently on these two orthogonal subspaces $\mathcal{R}_t$ and $\mathcal{R}_t^{\perp}$.
\begin{itemize}
\item \textbf{On the Range Space $\mathcal{R}_t$.}   Since the columns of $U_t$ forms an orthonormal basis of $\mathcal{R}_t$, the eigenvalues of $M_t$ restricted to this subspace are $\{1 - \alpha \sigma_i(J_F(\theta_t))\}_{i=1}^r$. By restricting the step size $0<\alpha\leq1/\sigma_{\max}$, we have
\begin{equation*}
    \max_{i} |1 - \alpha \sigma_i(J_F(\theta_t))| = 1 - \alpha \min_{i} \sigma_i(J_F(\theta_t)) \le 1 - \alpha\sigma_{\min}.
\end{equation*}
Consequently, the projection onto the range space satisfies the contraction property:
\begin{equation}\label{equ:contraction_range}
    \|M_t P_{\mathcal{R}_t} F(\theta_t)\|_2 \le (1 - \alpha\sigma_{\min}) \|P_{\mathcal{R}_t} F(\theta_t)\|_2 \le (1 - \alpha\sigma_{\min}) \|F(\theta_t)\|_2.
\end{equation}

    \item \textbf{On the Orthogonal Complement $\mathcal{R}_t^\perp$.} 
    By definition, any vector $v \in \mathcal{R}_t^\perp$ is orthogonal to the columns of $U_t$, which implies $U_t^{\top} v = 0$. Applying the operator $M_t$ to the projection $P_{\mathcal{R}_t^\perp} F(\theta_t)$ yields
    \begin{equation*}
        M_t P_{\mathcal{R}_t^\perp} F(\theta_t) = (I - \alpha U_t \Sigma_t U_t^{\top}) P_{\mathcal{R}_t^\perp} F(\theta_t) = P_{\mathcal{R}_t^\perp} F(\theta_t) .
    \end{equation*}
    Thus, $M_t$ acts as the identity operator on this subspace, preserving the norm
    \begin{equation}\label{equ:identity_null}
        \|M_t P_{\mathcal{R}_t^\perp} F(\theta_t)\|_2 = \|P_{\mathcal{R}_t^\perp} F(\theta_t)\|_2.
    \end{equation}
\end{itemize}
Combining~Lemma~4,~\cref{equ:estimation for Qt} and~\cref{equ:final relationship of F at t and t+1}, we have
\begin{equation}\label{equ:dynamics decomposition}
\begin{split}
    \|F(\theta_{t+1})\|_2 &\le \|M_t P_{\mathcal{R}_t} F(\theta_t)\|_2 + \|M_t P_{\mathcal{R}_t^\perp} F(\theta_t)\|_2 + \|Q_t\|_2 \\
    &\le (1 - \alpha\sigma_{\min}) \|F(\theta_t)\|_2 + C_{\perp} \|F(\theta_t)\|_2^2 + \frac{L_{J_F}\alpha^2}{2} \|F(\theta_t)\|_2^2.
\end{split}
\end{equation}
 By the induction hypothesis, $\theta_t \in W$. Hence, combining~\cref{equ:shrinking W} and~\cref{equ:dynamics decomposition} yields:
\begin{equation*}
    \|F(\theta_{t+1})\|_2 \le \left(1 - \alpha\sigma_{\min} + \frac{\alpha \sigma_{\min}}{2}\right) \|F(\theta_t)\|_2 = \left(1 - \frac{\alpha\sigma_{\min}}{2}\right) \|F(\theta_t)\|_2.
\end{equation*}
Using the induction hypothesis~\cref{equ:residual induction for thetat} again, we obtain
\begin{equation}\label{equ:residual induction for thetat+1}
    \|F(\theta_{t+1})\|_2 \le \left(1 - \frac{\alpha\sigma_{\min}}{2}\right)^{t+1} \|F(\theta_0)\|_2.
\end{equation}
Finally, summing the stepwise distances yields
\begin{equation}
\begin{split}
    \|\theta_{t+1} - \theta_0\|_2 &\le \sum_{k=0}^{t} \|\theta_{k+1} - \theta_k\|_2 
    \le \sum_{k=0}^{t} \alpha \|F(\theta_k)\|_2 \le \alpha \|F(\theta_0)\|_2 \sum_{k=0}^{t} \left(1 - \frac{\alpha\sigma_{\min}}{2}\right)^k \\
    &< \alpha \|F(\theta_0)\|_2 \cdot \frac{2}{\alpha\sigma_{\min}} 
    = \frac{2}{\sigma_{\min}} \|F(\theta_0)\|_2.
\end{split}
\end{equation}
By the construction~\cref{equ:definition of Ualpha_pgd} of $U_{\alpha},$ this implies $ \|\theta_{t+1}-\theta_0\|_2<\mathrm{dist}(\theta_0,\partial W)$, and therefore $\theta_{t+1}\in W$. Hence, the induction step holds.

\medskip
\noindent
 By induction, the sequence $\{\theta_t\}_{t \ge 0}$ generated by the SVD-preconditioned gradient descent with $\theta_0 \in U_{\alpha}$ remains in $W$ for all $t \ge 0$, and satisfies
\begin{equation*}
    \|F(\theta_t)\|_2 \le \left(1 - \frac{\alpha\sigma_{\min}}{2}\right)^t \|F(\theta_0)\|_2.
\end{equation*}
\end{proof}
\begin{remark}
With the step size chosen as $\alpha = 1/\sigma_{\max}$, the linear convergence factor of the proposed SPGD method becomes $1 - \sigma_{\min}/(2\sigma_{\max})$. By contrast, the local linear convergence factor of classical GD, established in~\cref{thm:local linear convergence of gd}, is given by $1 - \sigma_{\min}^{2}/\sigma_{\max}^{2}$. 
In the context of neural network optimization, it has been observed that the Jacobian $J_F(\theta)$ often becomes increasingly ill-conditioned as the network depth grows, primarily due to the monotonic rank-diminishing effect induced by layer composition~\cite{feng-2022-rank_diminishinig}. In such ill-conditioned regimes, the factor $1 - \sigma_{\min}^{2}/\sigma_{\max}^{2}$ associated with classical gradient descent rapidly approaches unity, leading to slow convergence. By comparison, the convergence factor $1 - \sigma_{\min}/(2\sigma_{\max})$ for the SPGD method deteriorates more slowly, resulting in a markedly improved local convergence behavior.

\end{remark}


\section{Adam framework with SPGD method}
\label{sec:adam framework with spgd}

We now extend the SPGD method to adaptive optimization frameworks. Specifically, we incorporate the SVD-based preconditioner into the Adam algorithm to leverage both curvature information from the Jacobian and adaptive moment estimates from Adam. This hybrid approach accelerates convergence while maintaining numerical stability, particularly in ill-conditioned or rank-deficient settings common in nonlinear least-squares problems, making it suitable for large-scale optimization tasks.

\subsection{Modified Adam algorithm}
The descent direction in~\cref{equ:update for theta at iteration t} can be equivalently expressed in a preconditioned gradient form as
\begin{equation*}
    \begin{split}
        V_tU_t^{\top}F(\theta_t)&=\bigl[(J_F^{\top}(\theta_t)J_F(\theta_t))^{\dagger}\bigr]^{\frac{1}{2}} J_{F}^{\top}(\theta_t)F(\theta_t)
=\bigl[(J_F^{\top}(\theta_t)J_F(\theta_t))^{\dagger}\bigr]^{\frac{1}{2}}\nabla f(\theta_t)
        \\
        &= B_t\nabla f(\theta_t),\hbox{~where~}B_t=\bigl[(J_F^{\top}(\theta_t)J_F(\theta_t))^{\dagger}\bigr]^{1/2}.
    \end{split}
\end{equation*} 
To further enhance convergence, we incorporate this preconditioned descent direction into the Adam optimization framework, resulting in the SPGD-Adam algorithm, which is summarized in~\cref{alg:svd-adam_algorithm}.

\begin{remark}[Extension to cross-entropy loss]\label{rem:cross_entropy_extension}
For a batch of $B$ samples and $K$ classes, the cross-entropy loss is defined as
\begin{equation*}
    \mathcal{L}_{\mathrm{CE}}(\theta)
    = -\frac{1}{B}\sum_{i=1}^{B}\sum_{j=1}^{K} y_{ij}\log(p_{ij}),
\end{equation*}
where $p_{ij}$ denotes the softmax probability corresponding to class $j$ for sample $i$, and
$F(\theta)\in\mathbb{R}^{B\times K}$ denotes the network output.
In this setting, the preconditioner can be generalized to
$    B_t = \bigl(J_F^{\top} C J_F + \delta I\bigr)^{-1/2},
$
where $\delta>0$ is a damping parameter and
$C\in\mathbb{R}^{BK\times BK}$ is a block-diagonal matrix whose $i$th block
$C_i\in\mathbb{R}^{K\times K}$ is given by
$    C_i = \mathrm{diag}(p_i) - p_i p_i^{\top}.
$
Each matrix $C_i$ is the Fisher information matrix of the categorical (softmax)
distribution associated with sample $i$ and is symmetric positive semidefinite \cite{chang2026fisher}.

\end{remark}

\begin{remark}[Efficient computation for large-scale networks]\label{rem:lanczos_approximation}
The direct computation of the preconditioner $B_t$ is computationally infeasible for large-scale networks, as it requires forming and factorizing the matrix
$J_F^{\top}J_F \in \mathbb{R}^{m\times m}$ (or $J_F^{\top} C J_F + \delta I$ in the cross-entropy setting), where $m$â€”the number of trainable parametersâ€”can range from tens of thousands to millions.
To overcome this bottleneck, we employ the Lanczos iteration method~\cite{golub-2013-matrix_computations} to efficiently approximate the action of $B_t$ on a vector without explicitly forming the matrix. Given a gradient vector $g$ and a prescribed number of Lanczos steps $k$ (with $k \ll m$), the method constructs a $k$-dimensional Krylov subspace
$\mathcal{K}_k(g) = \mathrm{span}\{g,\, Ag,\, A^2 g,\, \ldots,\, A^{k-1} g\},
A = J_F^{\top}J_F + \mu I,
$
(or $A = J_F^{\top} C J_F + \delta I$ in the cross-entropy case), and reduces the evaluation of the matrix function $A^{-1/2}g$ to a small tridiagonal problem.
Specifically, the Lanczos approximation takes the form
$    B_t g \;\approx\; \|g\|\, Q_k (T_k + \mu I)^{-1/2} e_1,
$
where $Q_k \in \mathbb{R}^{m\times k}$ contains the orthonormal Lanczos basis vectors,
$T_k \in \mathbb{R}^{k\times k}$ is the associated symmetric tridiagonal matrix,
and $e_1 = [1,0,\ldots,0]^{\top} \in \mathbb{R}^k$.
The matrix function $(T_k + \mu I)^{-1/2}$ is computed via eigendecomposition of the small $k\times k$ matrix, which is computationally inexpensive.

A key advantage of this approach is that Lanczos iteration only requires matrixâ€“vector products of the form $Av$, which can be computed efficiently using automatic differentiation (via Jacobian--vector and vector--Jacobian products), without explicitly forming $J_F$ or $J_F^{\top}J_F$.
As a result, the per-iteration computational cost scales as $O(k\,\mathrm{cost}(Av))$, making the method practical for networks with hundreds of thousands or even millions of parameters.
\end{remark}

\begin{algorithm}
\caption{SPGD method in Adam Framework}\label{alg:svd-adam_algorithm}
\begin{algorithmic}[1]
\setlength{\itemsep}{0.15ex}
    \Require Initial point $\theta_1$, hyperparameters $\beta_1,\beta_2$, step size $\alpha$
    \State $m_0\gets 0,\; v_0\gets 0$
    \For{$t=1,2,\ldots,T$}
        \State $J_t =  J_F(\theta_t)$,\quad
 $B_t = \bigl[(J_t^{\top}J_t)^{\dagger}\bigr]^{1/2}$, $g_t = \nabla f(\theta_t),\quad\lambda_t=B_tg_t$
        \State $m_t = \beta_1 m_{t-1} + (1-\beta_1)\lambda_t$, $v_t = \beta_2 v_{t-1} + (1-\beta_2)\lambda_t^{\odot2}$,\quad $ D_t = \operatorname{diag}(v_t)$
        \State $\theta_{t+1} = \theta_t - \alpha D_t^{-1/2} m_t$
    \EndFor
\end{algorithmic}
\end{algorithm}

\subsection{Global convergence}
In this section, we present the proof of global convergence for a modified version of our proposed~\cref{alg:svd-adam_algorithm}. The necessary assumptions and several auxiliary lemmas are provided in~\cref{subsubsec:assump} and Appendix~B respectively.

The original Adam optimizer proposed by Kingma and Ba in~\cite{kingma-2017-adam} is known to suffer from non-convergence even in the setting of online convex optimization~\cite{reddi-2018-AMSGrad}. To overcome this issue caused by the oscillatory term, Reddi et al.~\cite{reddi-2018-AMSGrad} introduced the AMSGrad variant, which eliminates the oscillatory behavior that hinders convergence. Following this idea, we adapt the AMSGrad framework~\cite{reddi-2018-AMSGrad,zhou-2024-convergence_of_adagrad} to modify our original~\cref{alg:svd-adam_algorithm}, thereby achieving the global convergence of the resulting Adam-like algorithm.

 For analytical tractability in our convergence analysis, we adopt a generalized and regularized form of the preconditioner, 
\begin{equation*}
    B_t=(J_t^{\top}J_t+\mu I)^{-p},
\end{equation*} 
where $\mu>0$ ensures $B_t$ is invertible. In addition, following the AMSGrad strategy~\cite{reddi-2018-AMSGrad,zhou-2024-convergence_of_adagrad}, we introduce the modified second moment estimate $\widehat{v}_t=\max\{\widehat{v}_{t-1},v_t\}$, and define the diagonal preconditioner as
\begin{equation}\label{equ:construction of widehat V}
    \widehat{D}_t=\operatorname{diag}(\widehat{v}_t+\epsilon).
\end{equation}
This construction~\cref{equ:construction of widehat V} guarantees that the sequence $\{\widehat{D}_t\}_{t\geq 1}$ forms a non-decreasing sequence of positive definite matrices, an essential property for proving convergence. Accordingly, the resulting update rule is given by 
\begin{equation*}
    \theta_{t+1}=\theta_t-\alpha\widehat{D}_t^{-\frac{1}{2}}m_t.
\end{equation*}

To prove the global convergence for the modified version of Algorithm~\ref{alg:svd-adam_algorithm}, rather than directly applying the descent lemma to the parameter sequence $\{\theta_t\}_{t\geq 1}$,
we follow the analysis framework of \cite{yang-2016-unified_convergence_analysis_stochastic,zhou-2024-convergence_of_adagrad}, and introduce the following auxiliary sequence $\{z_n\}_{n\geq 1}$ by 
    \begin{equation}\label{equ:the construction of z_t}
         z_t=\theta_t+\dfrac{\beta_1}{1-\beta_1}(\theta_t - \theta_{t-1})=\frac{1}{1-\beta_1}\theta_t - \frac{\beta_1}{1-\beta_1}\theta_{t-1}.
    \end{equation}
Here we adopt the convention that $\theta_0=\theta_1$.

Throughout this section, all quantities, such as $\theta_t,\lambda_t, m_t,v_t$ and $\widehat{v}_t$,
 are generated by the modified version of Algorithm~\ref{alg:svd-adam_algorithm} described above. For notational simplicity, we refer to it as the modified Algorithm~\ref{alg:svd-adam_algorithm} in the statements that follow.
\subsection{Assumptions}
\label{subsubsec:assump}
In this subsection, we introduce several assumptions that are essential in establishing the \cref{thm:convergence of the modified algorithm}. Recall that $f$ is the objective function defined in~\cref{equ:objective function}.

\begin{assumption}\label{assump:bounded gradient}
    $f$ has a bounded gradient, i.e., there exists a constant $M$, such that
           \begin{equation*}
               \|\nabla f(x)\|_{\infty}\leq M,\quad\forall x\in\mathbb{R}^m.
           \end{equation*}
\end{assumption}

\begin{assumption}\label{assump:L-smoothness}
    $f$ is $L$-smooth, i.e.,
           \begin{equation*}\label{gradient f is Lipschitz continuous}
               \|\nabla f(x)-\nabla f(y)\|\leq L\|x-y\|,\quad\forall x, y\in\mathbb{R}^m.
           \end{equation*}
\end{assumption}

\begin{assumption}\label{assump:bounded preconditioner}
The preconditioner $B_t=(J_t^{\top}J_t+\epsilon I)^{-p}$ is uniformly bounded and has a uniformly bounded inverse.
\end{assumption}

\begin{remark}\label{rem:notation for the upper bound of the matrix norm}
$\quad$
\begin{enumerate}
    \item[(i)] Since all norms on a finite-dimensional space are equivalent, the property of being uniformly bounded and having a uniformly bounded inverse is independent of the chosen norm.
    \item[(ii)] More precisely, for any given matrix norm $\|\cdot\|_{\star}$, there exists a constant $C_{\star}$  such that 
\begin{equation*}\label{equ:upper bound for the preconditioner}
    \|B_{t}\|_{\star}\leq C_{\star},\quad \|B_{t}^{-1}\|_{\star}\leq C_{\star}.
\end{equation*}
In particular, we denote the corresponding constant by $C_2$ for the spectral norm $\|\cdot\|_2$. Moreover, $B_t$ defines a linear operator acting on matrices via left multiplication, $A\to B_tA$, where the space of matrices is endowed with the matrix norm $\|\cdot\|_{\star}$. In this case, we further assume that the corresponding induced operator norms satisfy
\begin{equation}\label{equ:upper bound for operator norm of the preconditioner}
    \|B_t\|_{\star}^{\mathrm{op}}\leq C_{\star}^{\mathrm{op}},\quad \|B_t^{-1}\|_{\star}^{\mathrm{op}}\leq C_{\star}^{\mathrm{op}}.
\end{equation}
     \item [(iii)] Notice that $\lambda_t=B_tg_t=B_t\nabla f(\theta_t)$, we obtain
     \begin{equation}\label{equ:lambdat is bounded}
         \|\lambda_t\|_{\infty}\leq \|B_t\|_{\infty}\|\nabla f(\theta_t)\|_{\infty}\leq C_{\infty}M,
     \end{equation}
     which indicates that the preconditioned direction remains uniformly bounded.
\end{enumerate}
\end{remark}

With the preceding lemmas in place, we are prepared to present the proof of \cref{thm:convergence of the modified algorithm}.

\subsection{Convergence theorem}
 We are now ready to present the main convergence result for the modified Algorithm~\ref{alg:svd-adam_algorithm}. Our analysis builds on the framework of~\cite{zhou-2024-convergence_of_adagrad}, while incorporating several key modifications required by the SVD-based preconditioning used in our algorithmic update.

\begin{theorem}\label{thm:convergence of the modified algorithm}
  Consider the modified Algorithm~\ref{alg:svd-adam_algorithm} with hyperparameters $\beta_1,\beta_2$ and constant step size $\alpha$. Assume that 
$      \beta_1<\sqrt{\beta_2},\alpha_t=\alpha.
 $  Suppose that the gradient sequence  $\{g_t\}_{t=1}^T$ generated by the modified Algorithm~\ref{alg:svd-adam_algorithm} satisfies
$      \|g_{1:T,i}\|_2\leq M T^{s},\quad\forall\, T\geq 1, 0<s\leq\frac{1}{2}.
 $  Then the sequence of parameters $\{\theta_t\}_{t\geq 1}$ produced by the modified Algorithm~\ref{alg:svd-adam_algorithm} satisfies
  \begin{equation}\label{equ:convergence bound in main theorem}
    \begin{split}
        \frac{1}{T-1}\sum_{t=2}^T\|\nabla f(\theta_t)\|_2^2&\leq\frac{R_1}{(T-1)\alpha}+ \frac{R_2}{T-1}+{R_3}\frac{T^{\frac{3}{4}+\frac{s}{2}}}{T-1}+R_4\frac{T^{\frac{1}{2}+s}}{T-1}\alpha,
    \end{split}
  \end{equation}
where
\begin{equation*}
\begin{split}
\Delta_f&=f(\theta_1)-\inf_{\theta\in\mathbb{R}^m}f(\theta),~
R_1=\left(C_2^{\mathrm{op}})^2 \sqrt{C_{\infty}^2M^2+\epsilon}\right)\Delta_f,
\\
R_2 &= mM\left[\frac{1}{\sqrt{1-\beta_2}}+\frac{\beta_1MC_{1,1}^{\mathrm{op}}}{(1-\beta_1)\sqrt{\epsilon}}\right](C_2^{\mathrm{op}})^2 \sqrt{C_{\infty}^2M^2+\epsilon}
\\
R_3&=\sqrt{m}\left[\left(\frac{M\beta_1}{\sqrt{1-\beta_1}}+1\right)\sqrt{MC_2K}(1+C_2)\right]\left((C_2^{\mathrm{op}})^2 \sqrt{C_{\infty}^2M^2+\epsilon}\right),
\\
R_4&= mMKLC_2\left(\frac{C_2\beta_1(2+\beta_1)}{2(1-\beta_1)} + \frac{\beta_1^2-\beta_1+1}{(1-\beta_1)}\right)\left((C_2^{\mathrm{op}})^2 \sqrt{C_{\infty}^2M^2+\epsilon}\right).
\end{split}
\end{equation*}
\end{theorem}

\begin{proof}
Rather than applying Lemma~6 directly to the primary iterates $\{\theta_t\}$, as in conventional convergence analyses, we apply it to the auxiliary sequence $\{z_t\}_{t\geq 1}$ defined in \cref{equ:the construction of z_t}, following the approach in~\cite{zhou-2024-convergence_of_adagrad}:
   \begin{equation}\label{equ:applying descent lemma to z_t}
   \begin{split}
       f(z_{t+1})&\leq f(z_t)+\langle\nabla f(z_t),z_{t+1}-z_t\rangle+\frac{L}{2}\|z_{t+1}-z_t\|_2^2.
   \end{split} 
 \end{equation}
We decompose the inner product term on the right-hand side of \cref{equ:applying descent lemma to z_t} as follows:
   \begin{equation*}
   \begin{split}
         \langle\nabla f(z_t), z_{t+1}-z_t\rangle&=\langle(I-B_t)\nabla f(z_t), z_{t+1}-z_t\rangle+\langle B_t\nabla f(z_t)-B_t\nabla f(\theta_t), z_{t+1}-z_t\rangle\\&\quad+\langle B_t\nabla f(\theta_t), z_{t+1}-z_t\rangle. 
   \end{split}
   \end{equation*}
Thus, we obtain
   \begin{equation}\label{Four terms decomposition}
   \begin{split}
       f(z_{t+1})&\leq f(z_t)+ \underbrace{\langle(I-B_t)\nabla f(z_t), z_{t+1}-z_t\rangle}_{I_{1,t}}+\underbrace{\langle B_t\nabla f(z_t)-B_t\nabla f(\theta_t), z_{t+1}-z_t\rangle}_{I_{2,t}}\\&\quad+\underbrace{\langle B_t\nabla f(\theta_t), z_{t+1}-z_t\rangle}_{I_{3,t}}+\underbrace{\frac{L}{2}\|z_{t+1}-z_t\|_2^2}_{I_{4,t}}.
   \end{split}  
   \end{equation}
Next, we analyze each term in order to establish convergence.

\paragraph{Bounding $I_{1,t}$} By applying the Cauchy-Schwartz inequality, we obtain
\begin{equation}\label{equ:analysis of I_{1,t} first}
    \begin{split}
        I_{1,t}&\leq \|(I-B_t)\nabla f(z_t)\|_2\|z_{t+1}-z_t\|_2
        \leq \sup_{t\geq 1}\|(I-B_t)\|_{2}\cdot\sup_{z\in\mathbb{R}^m}\|\nabla f(z)\|_
        2\cdot\|z_{t+1}-z_t\|_2.
    \end{split}
\end{equation}
By Lemma~8, we have
\begin{equation}\label{equ:analysis of I_{1,t} second}
    \begin{split}
        I_{1,t} 
        &\leq (1+C_2)M\left(\frac{\beta_1}{1-\beta_1}\|\theta_t-\theta_{t-1}\|_2+\|A_t\lambda_t\|_2\right).
    \end{split}
\end{equation}
Here $C_2$ is the corresponding upper bound as defined in Remark~\ref{rem:notation for the upper bound of the matrix norm}. Summing \eqref{equ:analysis of I_{1,t} second} over time $t=2,\ldots,T$ yields

\begin{equation}\label{equ:final bound for t1}
\begin{split}
     \sum_{t=2}^{T}I_{1,t}&\leq \frac{M\beta_1(1+C_{2})}{1-\beta_1}\sum_{t=2}^{T}\|\theta_t-\theta_{t-1}\|_2+(1+C_{2})\sum_{t=2}^T\|A_t\lambda_t\|_2
     \\
     &=\frac{M\beta_1(1+C_2)}{1-\beta_1}\sum_{t=1}^{T-1}\|A_{t}m_{t}\|_2 +(1+C_2)\sum_{t=2}^T\|A_{t}\lambda_t\|_2.
\end{split}
\end{equation}

\paragraph{Bounding $I_{2,t}$} Applying the Cauchy-Schwartz  inequality and Assumption~\ref{assump:L-smoothness} yields
\begin{equation}
    \begin{aligned}\label{equ:analysis of I_{2,t} first}
      \left|I_{2,t}\right|&\leq\|B_{t}\|_2\|\nabla f(z_t)-\nabla f(\theta_t)\|_2\|z_{t+1}-z_t\|_2
            \leq  C_{2}L\|z_t-\theta_t\|_2\|z_{t+1}-z_t\|_2.   
\end{aligned}
\end{equation}
 Using the definition of $z_t$ in~\cref{equ:the construction of z_t} together with Lemma~8, we have
\begin{equation}\label{equ:analysis of I_{2,t} second}
\begin{aligned}
    |I_{2,t}| 
    &\leq C_{2}L\left(\frac{\beta_1}{1-\beta_1}\right) \|\theta_{t}-\theta_{t-1}\|_2\left(\frac{\beta_1}{1-\beta_1}\|\theta_t-\theta_{t-1}\|_2+\|A_t\lambda_t\|_2\right)
         \\
         &\leq C_{2} L\left(\frac{\beta_1}{1-\beta_1}\right)^2\|\theta_{t}-\theta_{t-1}\|_2^2+\frac{C_{2}L}{2}\left(\frac{\beta_1}{1-\beta_1}\right)(\|\theta_t-\theta_{t-1}\|_2^2+\|A_t\lambda_t\|_2^2).
\end{aligned}
\end{equation}

Summing~\cref{equ:analysis of I_{2,t} second} over time $t=2,\ldots, T$, we obtain
\begin{equation}\label{equ:final bound for t2}
\begin{split}
    \sum_{t=2}^T I_{2,t}&\leq \frac{C_2L}{2}\left(\frac{\beta_1}{1-\beta_1}\right)\sum_{t=2}^{T}\|\theta_t-\theta_{t-1}\|_2^2+\frac{C_2L}{2}\left(\frac{\beta_1}{1-\beta_1}\right)\sum_{t=2}^T\|A_t\lambda_t\|_2^2
    \\
    &\quad+ C_2L\left(\frac{\beta_1}{1-\beta_1}\right)^2\sum_{t=2}^T\|\theta_t-\theta_{t-1}\|_2^2
    \\
    &=\frac{C_2L}{2}\frac{\beta_1(1+\beta_1)}{(1-\beta_1)^2}\sum_{t=1}^{T-1}\|A_{t}m_{t}\|_2^2+\frac{C_2L}{2}\left(\frac{\beta_1}{1-\beta_1}\right)\sum_{t=2}^T\|A_t\lambda_t\|_2^2.
\end{split}
\end{equation}

\paragraph{Decomposing $I_{3,t}$}
By using Lemma~7, we have
\begin{equation}\label{equ:analysis of I_{3,t}}
        \begin{split}
            I_{3,t}&=\langle B_t\nabla f(\theta_t),z_{t+1}-z_t\rangle=\left\langle B_t\nabla f(\theta_t),\frac{\beta_1}{1-\beta_1}A_{t-1}m_{t-1}-\frac{\beta_1}{1-\beta_1}A_tm_{t-1}-A_t\lambda_t\right\rangle\\            &=\underbrace{\frac{\beta_1}{1-\beta_1}\left\langle B_t\nabla f(\theta_t), A_{t-1}m_{t-1}-A_tm_{t-1}\right\rangle}_{I_{3,t}^{(A)}} -\underbrace{\langle B_t\nabla f(\theta_t), A_t\lambda_t\rangle}_{I_{3,t}^{(B)}}
        \end{split}
    \end{equation}
For the first part, we have
\begin{equation}\label{equ:analysis of I_{3,t}A first}
\begin{aligned}
     \left|I_{3,t}^{(A)}\right|
     &\leq \frac{\beta_1}{1-\beta_1}\sup_{t\geq 1}\|\nabla f(\theta_t)\|_{\infty}\cdot\|B_t(A_{t-1}-A_t)\|_{1,1}\cdot\sup_{t\geq1}\|m_{t-1}\|_{\infty}
     \\
     &\leq\frac{\beta_1M^2}{1-\beta_1}\cdot\sup_{t\geq 1}\|B_t\|_{1,1}^{\mathrm{op}}\cdot\|A_{t-1}-A_t\|_{1,1},
\end{aligned}
\end{equation}
where $\|A_t\|_{1,1}=\sum_{i,j}|A_{t,ij}|$ and ${A_{t,ij}}$ denotes the $(i,j)$-th entry of $A_t$.
Due to the construction in~\cref{equ:construction of widehat V}, the sequence $\{\widehat{D}_t\}_{t\geq 1}$ is non-decreasing in the sense of positive semi-definiteness. Consequently, the sequence $\{A_t\}_{t\geq 1}$ defined as $A_t=\alpha\widehat{D}_t^{-1/2}$ is non-increasing and consists of diagonal matrices. Therefore,
\begin{equation}\label{equ:norm of difference of At}
    \|A_{t-1}-A_t\|_{1,1}=\|A_{t-1}\|_{1,1}-\|A_t\|_{1,1}.
\end{equation}
Combining~\cref{equ:analysis of I_{3,t}A first} and~\cref{equ:norm of difference of At}, we obtain
\begin{equation}\label{equ:analysis of I_{3,t}A second}
\begin{aligned}
     \left|I_{3,t}^{(A)}\right|
    &\leq \frac{\beta_1M^2}{1-\beta_1}\cdot\sup_{t\geq 1}\|B_t\|_{1,1}^{\mathrm{op}}(\|A_{t-1}\|_{1,1}-\|A_t\|_{1,1})\leq \frac{\beta_1M^2C_{1,1}^{\mathrm{op}}}{1-\beta_1}(\|A_{t-1}\|_{1,1}-\|A_t\|_{1,1}).
\end{aligned}
\end{equation}
Summing the bound in~\cref{equ:analysis of I_{3,t}A second} over time $t=2,\ldots,T$ leads to
    \begin{equation}\label{equ:final bound for t3a}
        \sum_{t=2}^TI_{3,t}^{(A)}\leq \frac{\beta_1M^2C_{1,1}^{\mathrm{op}}}{1-\beta_1}\|A_1\|_{1,1}=O(1).
    \end{equation}
For the second part, we have
  \begin{equation*}
    \begin{split}
        I_{3,t}^{(B)}&=\langle B_t\nabla f(\theta_
        t), A_t\lambda_t\rangle=\langle B_t\nabla f(\theta_t), A_tB_t\nabla f(\theta_t)\rangle=\big\| A_t^{\frac{1}{2}} B_t\nabla f(\theta_t)\big\|_2^2.
    \end{split}
    \end{equation*}
   To further relate $I_{3,t}^{(B)}$ to the squared gradient norm, we note that
    \begin{equation}\label{equ:gradient bounded by I_{3,t}}
        \begin{split}
        \|\nabla f(\theta_t)\|_2&=\big\|\big(A_t^{\frac{1}{2}}B_t\big)^{-1}\big(A_{t}^{\frac{1}{2}}B_t\big)\nabla f(\theta_t)\big\|_2\leq\sup_{t\geq 1}\big\|B_t^{-1}A_t^{-\frac{1}{2}}\big\|_2\cdot \sqrt{I_{3,t}^{(B)}}.
        \end{split}
    \end{equation}
By using ~Lemma~5, we have
\begin{equation}\label{equ:upper bound for BtAt}
    \sup_{t\geq 1}\big\|B_t^{-1}A_t^{-\frac{1}{2}}\big\|_2\leq \sup_{t\geq 1}\big\|B_t\big\|_2^{\mathrm{op}}\cdot\big\|A_t^{-\frac{1}{2}}\big\|_2\leq{C}_2^{\mathrm{op}}\cdot\sqrt[4]{\frac{C_{\infty}^2M^2+\epsilon}{\alpha^2}},
\end{equation}
Combining~\cref{equ:gradient bounded by I_{3,t}} and~\cref{equ:upper bound for BtAt} gives
\begin{equation}\label{equ:final upper bound of gradient by I_{3,t}B}
\|\nabla f(\theta_t)\|^2_2\leq \frac{(C_2^{\mathrm{op}})^2 \sqrt{C_{\infty}^2M^2+\epsilon}}{\alpha} I_{3,t}^{(B)}.
\end{equation}
Summing~\cref{equ:final upper bound of gradient by I_{3,t}B} over time $t=2,\ldots,T$ yields
\begin{equation}\label{equ:final bound for t3b}
    \begin{split}
    \frac{1}{T-1}\sum_{t=2}^T\|\nabla f(\theta_t)\|_2^2\leq \frac{(C_2^{\mathrm{op}})^2 \sqrt{C_{\infty}^2M^2+\epsilon}}{(T-1)\alpha}\sum_{t=2}^TI_{3,t}^{(B)}.
    \end{split}
\end{equation}

\paragraph{Bounding $I_{4,t}$} By ~Lemma~8 and the Cauchy-Schwartz inequality, we have.
   \begin{equation}\label{analysis of I_{4,t}}
       \begin{split}
           I_{4,t}           &\leq L\left(\frac{\beta_1}{1-\beta_1}\right)^2\|\theta_t-\theta_{t-1}\|_2^2+L\|A_t\lambda_t\|_2^2.
       \end{split}
   \end{equation}
Summing~\eqref{analysis of I_{4,t}} over time $t=2,\ldots, T$ implies that
\begin{equation}\label{equ:final bound for t4}
\begin{split}
     \sum_{t=2}^T I_{4,t}&\leq L\left(\frac{\beta_1}{1-\beta_1}\right)^2\sum_{t=2}^T\|\theta_t-\theta_{t-1}\|_2^2+L\sum_{t=2}^T \|A_t\lambda_t\|_2^2
     \\
     &= L\left(\frac{\beta_1}{1-\beta_1}\right)^2\sum
     _{t=1}^{T-1}\|A_{t}m_{t}\|_2^2+L\sum_{t=2}^T\|A_t\lambda_t\|_2^2.
\end{split}
\end{equation}

\paragraph{Summarization}

Therefore, we take the telescoping sum of \eqref{Four terms decomposition} for $t=2,\ldots, T$ and have    \begin{equation}\label{equ:telescoping sum}
        \begin{split}
            f(z_{T+1})-f(z_2)\leq \sum_{t=2}^TI_{1,t}+\sum_{t=2}^T I_{2,t}+\sum_{t=2}^T I_{3,t}^{(A)}-\sum_{t=2}^T I_{3,t}^{(B)}+\sum_{t=2}^T I_{4,t}.
        \end{split}
    \end{equation}
 For $t=1$, by Lemma~7, we have
 \begin{equation}\label{equ:special case}
     \begin{split}
         f(z_2)-f(z_1)
         &\leq \langle\nabla f(z_1), z_2-z_2\rangle+\frac{L}{2}\|z_2-z_1\|_2^2\leq -\langle\nabla f(\theta_1),A_1\lambda_1\rangle+\frac{L}{2}\|A_1\lambda_1\|_2^2.
     \end{split}
 \end{equation}
Since\[ |\langle\nabla f(\theta_1),A_1\lambda_1\rangle|\leq m\|\nabla f(\theta_1)\|_{\infty}\|A_1\lambda_1\|_{\infty}\leq  mM\big\|V_1^{-\frac{1}{2}}\lambda_1\big\|_{\infty}\alpha
\leq\frac{mM\alpha}{\sqrt{1-\beta_2}},\] combining~\cref{equ:telescoping sum} and~\cref{equ:special case} gives
 \begin{equation}
     \begin{split}
            f(z_{T+1})-f(z_1)\leq \sum_{t=2}^TI_{1,t}+\sum_{t=2}^T I_{2,t}+\sum_{t=2}^T I_{3,t}^{(A)}-\sum_{t=2}^T I_{3,t}^{(B)}+\sum_{t=1}^T I_{4,t}+\frac{mM\alpha}{\sqrt{1-\beta_2}}.
        \end{split}
 \end{equation}
 Employing~\cref{equ:final bound for t3b} and rearranging terms of~\cref{equ:telescoping sum} gives
\begin{equation*}
    \begin{split}
       \frac{1}{T-1}\sum_{t=2}^T\|\nabla f(\theta_t)\|_2^2
       &\leq \frac{(C_2^{\mathrm{op}})^2 \sqrt{C_{\infty}^2M^2+\epsilon}}{(T-1)\alpha}\sum_{t=2}^TI_{3,t}^{(B)}\leq\frac{(C_2^{\mathrm{op}})^2 \sqrt{C_{\infty}^2M^2+\epsilon}}{(T-1)\alpha}\left(f(z_1)-f(z_{T+1})+\frac{mM\alpha}{\sqrt{1-\beta_2}}\right)
        \\
        & +\frac{(C_2^{\mathrm{op}})^2 \sqrt{C_{\infty}^2M^2+\epsilon}}{(T-1)\alpha}\left(\sum_{t=2}^TI_{1,t} + \sum_{t=2}^T I_{2,t}+\sum_{t=2}^T I_{3,t}^{(A)}+\sum_{t=1}^T I_{4,t}\right).
        \\
        &
    \end{split}
\end{equation*}
Combining all the bounds~\cref{equ:final bound for t1}, \cref{equ:final bound for t2}, \cref{equ:final bound for t3a} and~\cref{equ:final bound for t4} from previous steps, we obtain
\begin{equation}\label{equ:bound by variation and delta}
\begin{split}
\frac{1}{T-1}\sum_{t=2}^T\|\nabla f(\theta_t)\|_2^2 
&\leq \frac{(C_2^{\mathrm{op}})^2 \sqrt{C_{\infty}^2M^2+\epsilon}}{(T-1)\alpha}\Bigg[f(z_1)-f(z_{T+1}) + \frac{mM\alpha}{\sqrt{1-\beta_2}}+\frac{\beta_1M^2C_{1,1}^{\mathrm{op}}}{1-\beta_1}\|A_1\|_{1,1} 
\\
& + \frac{M\beta_1(1+C_2)}{1-\beta_1}\sum_{t=1}^{T}\|A_t m_t\|_2 + (1+C_2)\sum_{t=1}^T\|A_t\lambda_t\|_2 
\\
& + \frac{\beta_1L}{2(1-\beta_1)^2}(C_2\beta_1+2\beta_1+C_2)\sum_{t=1}^{T}\|A_t m_t\|_2^2 + \frac{L}{2(1-\beta_1)}(C_2\beta_1-2\beta_1+2)\sum_{t=1}^T\|A_t\lambda_t\|_2^2\Bigg].
\end{split}
\end{equation}
Applying~Lemma~11 and~Lemma~12, we obtain
\begin{equation}
 \begin{split}
\frac{1}{T-1}\sum_{t=2}^T\|\nabla f(\theta_t)\|_2^2 
&\leq \frac{(C_2^{\mathrm{op}})^2 \sqrt{C_{\infty}^2M^2+\epsilon}}{(T-1)\alpha}\Bigg[
f(z_1)-f(z_{T+1}) 
+ \frac{mM\alpha}{\sqrt{1-\beta_2}}+\frac{\beta_1M^2C_{1,1}^{\mathrm{op}}}{1-\beta_1}\big\|V_1^{-\frac{1}{2}}\big\|_{1,1}\alpha 
\\
&+ \frac{M\beta_1(1+C_2)\sqrt{K}}{\sqrt{1-\beta_1}}\left(\sum_{i=1}^m\|\lambda_{1: T,i}\|_2\right)^{\frac{1}{2}}T^{\frac{3}{4}}\alpha
+ (1+C_2)\sqrt{K}\left(\sum_{i=1}^m\|\lambda_{1: T,i}\|_2\right)^{\frac{1}{2}} T^{\frac{3}{4}}\alpha
\\
& + \frac{KL\beta_1}{2(1-\beta_1)}(C_2\beta_1+2\beta_1+C_2)\left(\sum_{i=1}^m\|\lambda_{1: T,i}\|_2\right) T^{\frac{1}{2}}\alpha^2
\\
& + \frac{KL}{2(1-\beta_1)}(C_2\beta_1-2\beta_1+2)\left(\sum_{i=1}^m\|\lambda_{1: T,i}\|_2\right)T^{\frac{1}{2}}\alpha^2\Bigg].
\end{split}
\end{equation}
Using~Lemma~10 together with the fact that $\|g_{1:T,i}\|\leq MT^s$ holds for all $T\geq 1$ and $1\leq i\leq m$, we have
\begin{equation*}
    \begin{split}
        \frac{1}{T-1}\sum_{t=2}^T\|\nabla f(\theta_t)\|_2^2&\leq\frac{R_1}{(T-1)\alpha}+ \frac{R_2}{T-1}+{R_3}\frac{T^{\frac{3}{4}+\frac{s}{2}}}{T-1}+R_4\frac{T^{\frac{1}{2}+s}}{T-1}\alpha.
    \end{split}
\end{equation*}
\end{proof}

\section{Numerical Experiments}
\label{sec:numerical_experiments}

\subsection{Function approximation}
\label{subsec:function_approximation}
We evaluate the proposed SPGD method against Adam on a function regression problem with
\begin{equation}
    f(\mathbf{x}) = \sin\Bigl(n \pi \sum_{i=1}^d x_i \Bigr), 
    \quad \mathbf{x}=(x_1,\dots,x_d) \in [0,1]^d,
\end{equation}
approximated by a neural network $u(\mathbf{x};\theta)$. Two experimental scenarios are considered: 

\begin{itemize}
    \item \textbf{Scenario I:} varying the frequency parameter $n \in \{3,5,7,9\}$ with fixed input dimension $d=1$, testing the effect of high-frequency oscillations, which are known to slow learning due to the frequency principle~\cite{xu-2019-frequency-principle}.
    \item \textbf{Scenario II:} varying the input dimension $d \in \{3,5,7,9\}$ with fixed $n=1$, examining performance in high-dimensional search spaces subject to the curse of dimensionality.
\end{itemize}

The network $u(\mathbf{x};\theta)$ is a fully connected feedforward network with $4$ hidden layers of width $64$ and Tanh activations, initialized via Xavier normal initialization with biases set to zero. Training minimizes the $L^2$ loss
$    \mathcal{L}(\theta) = \int_{[0,1]^d} |u(\mathbf{x};\theta) - f(\mathbf{x})|^2 \,\mathrm{d}\mathbf{x},
$
approximated using mini-batch sampling (batch size $256$) at each epoch. A fixed test set of size $2^{15}$ is used to evaluate convergence. All experiments are run for $15{,}000$ epochs and repeated $10$ times with different random seeds. 

Both optimizers use Adam hyperparameters $\beta_1=0.9$, $\beta_2=0.999$, $\epsilon=10^{-8}$, and an initial learning rate of $10^{-2}$. Learning rates decay via a staircase schedule: for Adam, the rate is halved every $1000$ epochs; for SPGD, it is reduced by a factor of $0.7$ every $100$ epochs in Scenario~I and by a factor of $0.9$ every $100$ epochs in Scenario~II, with a minimum enforced learning rate of $10^{-5}$.

\cref{fig:freq_d1,fig:dim_n1} display the convergence histories for the two scenarios, with each subplot corresponding to a specific frequency or dimension. Median test losses and interquartile ranges across $10$ runs are shown, smoothed using a rolling mean. \cref{tab:freq_d1_final,tab:freq_d1_milestone,tab:dim_n1_final,tab:dim_n1_milestone} summarize final test loss statistics and milestone epochs required to reach prescribed thresholds.

The results indicate that SPGD consistently outperforms Adam. Specifically, SPGD converges substantially faster: median epochs to reach thresholds such as $10^{-4}$ or $10^{-5}$ are reduced by roughly an order of magnitude, with success rates near $100\%$, even for high-frequency or high-dimensional settings where Adam often fails. Furthermore, SPGD achieves lower final test loss: across all tested frequencies and dimensions, median, best, and worst losses are consistently one to two orders of magnitude smaller than those obtained by Adam. These observations confirm the effectiveness of the proposed preconditioning in handling both oscillatory and high-dimensional optimization challenges.

\begin{figure}[t]
    \centering
    \includegraphics[width=\linewidth]{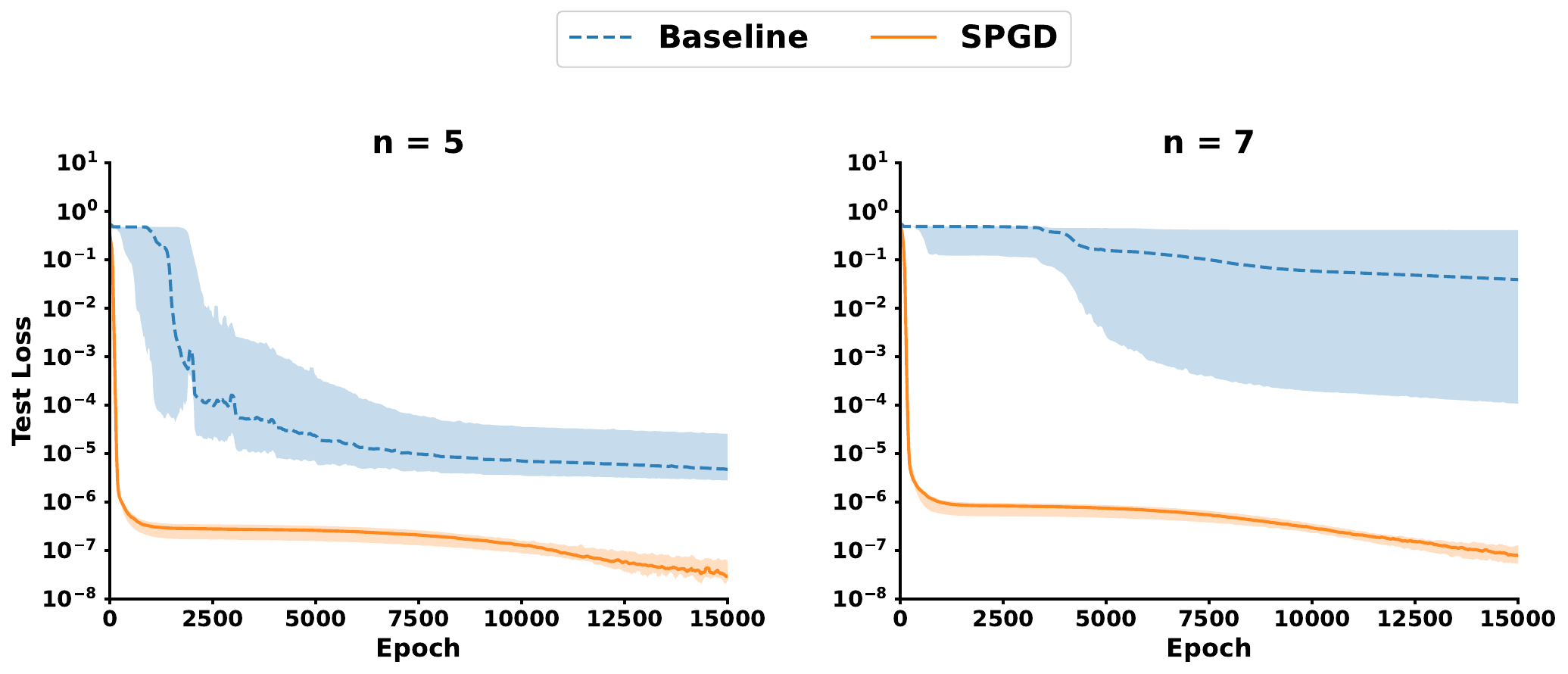}
    \caption{Scenario~I (Varying Frequency): Test loss versus the number of epochs with fixed dimension $d$ and varying frequency parameters $n \in \{5, 7\}$. The shaded region in each subplot indicates the interquartile range over 10 independent runs with different random seeds.}
    \label{fig:freq_d1}
\end{figure}

\begin{table}[t]
\centering\small
\caption{Scenario~I (Varying Frequency): Final test loss statistics comparing different frequency parameters $n$ over 10 independent runs.}
\label{tab:freq_d1_final}
\small
\begin{tabular}{c l c c c}
\toprule
$n$ & Optimizer & Median & Best & Worst \\
\midrule
3 & Adam & $1.28\times10^{-6}$ & $4.95\times10^{-7}$ & $2.45\times10^{-4}$ \\
3 & SPGD & $\mathbf{1.66\times10^{-8}}$ & $\mathbf{8.21\times10^{-9}}$ & $\mathbf{1.24\times10^{-7}}$ \\
\addlinespace
5 & Adam & $4.86\times10^{-6}$ & $1.40\times10^{-6}$ & $4.12\times10^{-4}$ \\
5 & SPGD & $\mathbf{3.16\times10^{-8}}$ & $\mathbf{1.20\times10^{-8}}$ & $\mathbf{8.48\times10^{-8}}$ \\
\addlinespace
7 & Adam & $3.93\times10^{-2}$ & $2.64\times10^{-6}$ & $4.87\times10^{-1}$ \\
7 & SPGD & $\mathbf{8.54\times10^{-8}}$ & $\mathbf{3.29\times10^{-8}}$ & $\mathbf{6.27\times10^{-7}}$ \\
\addlinespace
9 & Adam & $4.73\times10^{-1}$ & $6.25\times10^{-6}$ & $4.96\times10^{-1}$ \\
9 & SPGD & $\mathbf{2.52\times10^{-7}}$ & $\mathbf{1.25\times10^{-7}}$ & $\mathbf{7.13\times10^{-7}}$ \\
\bottomrule
\end{tabular}
\end{table}

\begin{table}[t]
\centering\small
\caption{Scenario~I (Varying Frequency): Milestone statistics (epochs and time) required to reach prescribed error thresholds. "Success" is defined as the percentage of independent runs that meet the threshold, and other values denote the median across multiple successful runs.}
\label{tab:freq_d1_milestone}
\small 
\setlength{\tabcolsep}{2.5pt}
\begin{tabular}{c l c c c c c c c c c}
\toprule
 &  & \multicolumn{3}{c}{Threshold $10^{-4}$}
 & \multicolumn{3}{c}{Threshold $10^{-5}$}
 & \multicolumn{3}{c}{Threshold $10^{-6}$} \\
\cmidrule(lr){3-5}\cmidrule(lr){6-8}\cmidrule(lr){9-11}
$n$ & Optimizer
& Epoch & Time & Success 
& Epoch & Time & Success 
& Epoch & Time & Success  \\
\midrule
3 & Adam
& 838 & \textbf{0.10} & 90\%
& 1597 & \textbf{0.20} & 90\%
& 5304 & 0.64 & 50\% \\
3 & SPGD
& \textbf{86} & 0.26 & \textbf{100}\%
& \textbf{108} & 0.33 & \textbf{100}\%
& \textbf{144} & \textbf{0.43} & \textbf{100}\% \\
\addlinespace
5 & Adam
& 1571 & \textbf{0.19} & 80\%
& 3314 & 0.40 & 60\%
& -- & -- & 0\% \\
5 & SPGD
& \textbf{104} & 0.31 & \textbf{100}\%
& \textbf{129} & \textbf{0.39} & \textbf{100}\%
& \textbf{241} & \textbf{0.73} & \textbf{100}\% \\
\addlinespace
7 & Adam
& 1392 & \textbf{0.17} & 30\%
& 3247 & \textbf{0.39} & 30\%
& -- & -- & 0\% \\
7 & SPGD
& \textbf{133} & 0.4 & \textbf{100}\%
& \textbf{176} & 0.53 & \textbf{100}\%
& \textbf{964} & \textbf{2.90} & \textbf{100}\% \\
\addlinespace
9 & Adam
& 1980 & \textbf{0.24} & 20\%
& 8196 & 0.99 & 20\%
& -- & -- & 0\% \\
9 & SPGD
& \textbf{151} & 0.46 & \textbf{100}\%
& \textbf{262} & \textbf{0.80} & \textbf{100}\%
& \textbf{8469} & \textbf{25.53} & \textbf{100}\% \\
\bottomrule
\end{tabular}
\end{table}

\begin{figure}[t]
    \centering
    \includegraphics[width=\linewidth]{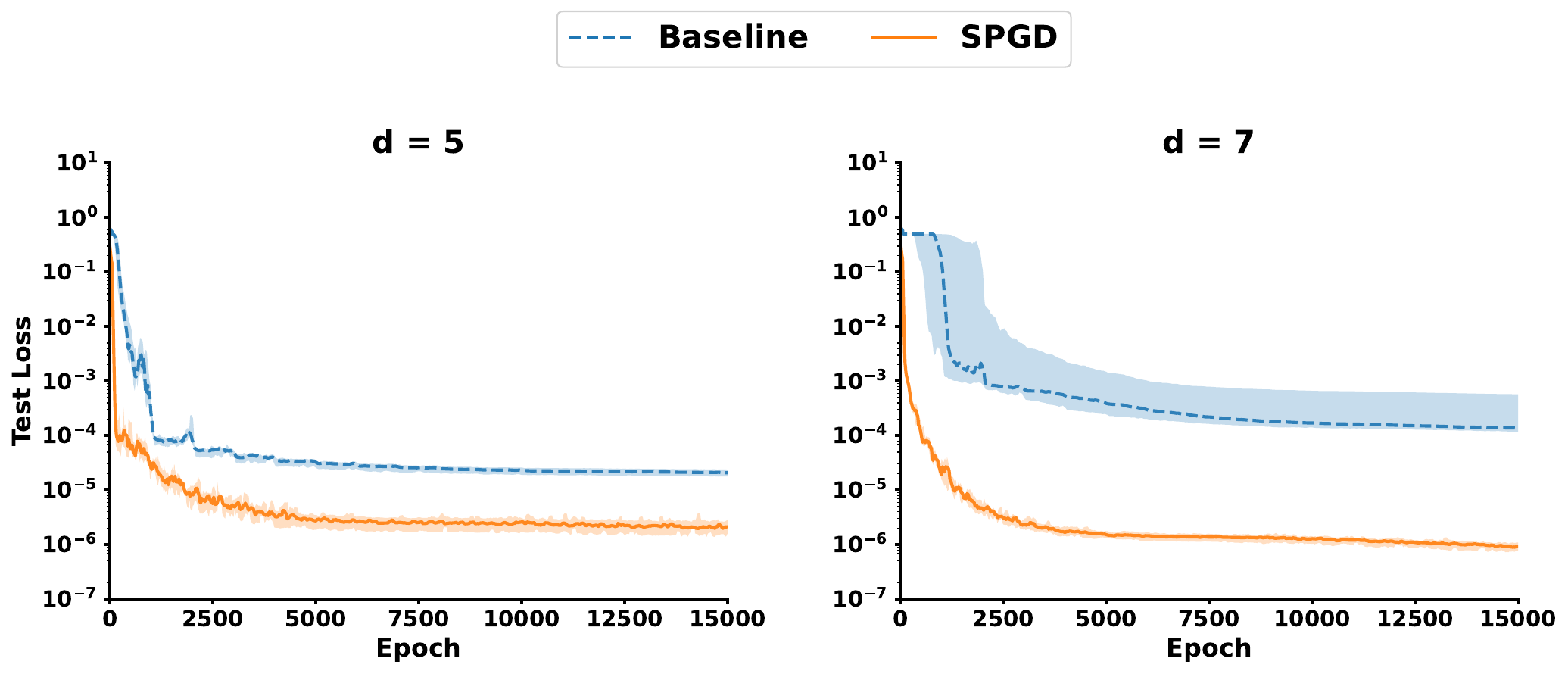}
    \caption{Scenario~II (Varying Dimension): Test loss versus the number of epochs with fixed frequency $n=1$ and varying input dimensions $d \in \{5, 7\}$. The shaded region in each subplot indicates the interquartile range over 10 independent runs with different random seeds.}
    \label{fig:dim_n1}
\end{figure}

\begin{table}[t]
\centering
\caption{Scenario~II (Varying Dimension): Final test loss statistics comparing different input dimensions $d$ over 10 independent runs.}
\label{tab:dim_n1_final}
\small
\begin{tabular}{c l c c c}
\toprule
$d$ & Optimizer & Median & Best & Worst \\
\midrule
3 & Adam 
& $7.12\times10^{-6}$ 
& $4.83\times10^{-6}$ 
& $1.01\times10^{-5}$ \\
3 & SPGD
& $\mathbf{8.43\times10^{-8}}$ 
& $\mathbf{5.72\times10^{-8}}$ 
& $\mathbf{1.90\times10^{-7}}$ \\
\addlinespace
5 & Adam 
& $2.08\times10^{-5}$ 
& $1.32\times10^{-5}$ 
& $1.58\times10^{-4}$ \\
5 & SPGD
& $\mathbf{2.13\times10^{-6}}$ 
& $\mathbf{6.07\times10^{-7}}$ 
& $\mathbf{5.35\times10^{-6}}$ \\
\addlinespace
7 & Adam 
& $1.37\times10^{-4}$ 
& $7.36\times10^{-5}$ 
& $6.70\times10^{-4}$ \\
7 & SPGD 
& $\mathbf{9.36\times10^{-7}}$ 
& $\mathbf{5.47\times10^{-7}}$ 
& $\mathbf{1.36\times10^{-6}}$ \\
\addlinespace
9 & Adam 
& $4.93\times10^{-1}$ 
& $7.06\times10^{-5}$ 
& $4.97\times10^{-1}$ \\
9 & SPGD
& $\mathbf{6.39\times10^{-6}}$ 
& $\mathbf{4.73\times10^{-6}}$ 
& $\mathbf{9.80\times10^{-6}}$ \\
\bottomrule
\end{tabular}
\end{table}

\begin{table}[t]
\centering
\caption{Scenario~II (Varying Dimension): Milestone statistics (epochs and time) required to reach prescribed error thresholds. "Success" is defined as the percentage of independent runs that meet the threshold, and other values denote the median across multiple successful runs.}
\label{tab:dim_n1_milestone}
\small
\setlength{\tabcolsep}{3pt}
\begin{tabular}{c l c c c c c c c c c}
\toprule
 &  & \multicolumn{3}{c}{Threshold $10^{-3}$}
 & \multicolumn{3}{c}{Threshold $10^{-4}$}
 & \multicolumn{3}{c}{Threshold $10^{-5}$} \\
\cmidrule(lr){3-5}\cmidrule(lr){6-8}\cmidrule(lr){9-11}
$d$ & Optimizer 
& Epoch & Time & Success 
& Epoch & Time & Success 
& Epoch & Time & Success \\
\midrule
3 & Adam 
& 470 & \textbf{0.06} & 100\%
& 1648 & \textbf{0.21} & 100\%
& 8902 & {1.13} & 100\% \\
3 & SPGD
& \textbf{67} & {0.21} & {100\%}
& \textbf{117} & 0.36 & {100\%}
& \textbf{316} & \textbf{0.95} & {100\%} \\
\addlinespace
5 & Adam 
& 463 & \textbf{0.06} & 100\%
& 1000 & \textbf{0.13} & 90\%
& -- & -- & -\% \\
5 & SPGD
& \textbf{72} & 0.23 & {100\%}
& \textbf{116} & {0.37} & \textbf{100\%}
& \textbf{1034} & \textbf{3.20} & \textbf{100\%} \\
\addlinespace
7 & Adam 
& 1790 & \textbf{0.23} & 100\%
& 7862 & \textbf{1.00} & 20\%
& -- & -- & 0\% \\
7 & SPGD
& \textbf{117} & {0.37} & {100\%}
& \textbf{384} & {1.22} & \textbf{100\%}
& \textbf{998} & \textbf{3.15} & \textbf{100\%} \\
\addlinespace
9 & Adam 
& 5266 & 0.69 & 20\%
& 10898 & 1.44 & 10\%
& -- & -- & 0\% \\
9 & SPGD
& \textbf{138} & \textbf{0.45} & \textbf{100\%}
& \textbf{208} & \textbf{0.67} & \textbf{100\%}
& \textbf{2739} & \textbf{8.63} & \textbf{100\%} \\
\bottomrule
\end{tabular}
\end{table}

\subsection{Poisson equation}
\label{subsec:poisson_equation}
We compare the proposed SPGD (\cref{alg:svd-adam_algorithm}) with Adam on the Poisson problem
\begin{equation}
\label{equ:laplacian_problem}
\begin{cases}
-\Delta u = -2d, & x \in \Omega,\\
u = 1, & x \in \partial\Omega,
\end{cases}
\end{equation}
posed on the $d$-dimensional unit ball
$\Omega=\{x\in\mathbb{R}^d:\|x\|_2\leq1\}$.
The exact solution is given by $u^*(x)=\|x\|_2^2$.
We approximate the solution using a neural network $u(\mathbf{x};\theta)$ trained to minimize a weighted sum of interior and boundary residuals,
\begin{equation}
\mathcal{L}(\theta)
=
\frac{1}{N_{\mathrm{pde}}}\sum_{i=1}^{N_{\mathrm{pde}}}
\bigl(\Delta u(\mathbf{x}_i;\theta) - 2d \bigr)^2
+
\lambda_{\mathrm{bc}}\,
\frac{1}{N_{\mathrm{bc}}}\sum_{b=1}^{N_{\mathrm{bc}}}
\bigl(u(\mathbf{x}_b;\theta) - 1\bigr)^2.
\end{equation}

The network consists of $4$ hidden layers of width $16$ with $\mathrm{ReLU}^3$ activation functions, initialized using Xavier normal initialization with zero biases. At each training epoch, mini-batches of $N_{\mathrm{pde}}=256$ interior points and $N_{\mathrm{bc}}=256$ boundary points are sampled uniformly from $\Omega$ and $\partial\Omega$, respectively. A fixed test set of size $2^{15}$ is used to evaluate the relative $L^2$ error. All experiments are run for $10{,}000$ epochs and repeated $10$ times with different random seeds for dimensions $d\in\{2,4,6,8\}$.

Both optimizers use Adam hyperparameters $\beta_1=0.9$, $\beta_2=0.999$, and $\epsilon=10^{-8}$, with an initial learning rate of $3\times10^{-3}$. A staircase decay schedule is employed: for Adam, the learning rate is halved every $1000$ epochs, while for SPGD it is multiplied by $0.9$ every $100$ epochs, with a minimum enforced learning rate of $3\times10^{-6}$. The boundary penalty parameter is fixed at $\lambda_{\mathrm{bc}}=1000$.

\cref{fig:loss_d4_d6} presents representative convergence histories for dimensions $d=4$ and $d=6$, showing the median relative $L^2$ error and interquartile range across $10$ runs (smoothed using a rolling mean). \cref{tab:loss_pde,tab:milestone_pde} report final relative $L^2$ error statistics and milestone epochs required to reach prescribed accuracy thresholds for all tested dimensions.

The results demonstrate that SPGD consistently outperforms Adam as the spatial dimension increases. In particular, SPGD converges substantially faster: the median number of epochs required to reach error thresholds such as $10^{-1}$ or $10^{-2}$ is reduced by a factor of $3$--$5$ across all dimensions, enabling robust convergence in higher-dimensional settings where Adam exhibits significant slowdowns. Moreover, SPGD achieves lower final accuracy: the median relative $L^2$ errors are approximately one order of magnitude smaller than those obtained by Adam across all tested dimensions. These results confirm the stability and efficiency of the proposed spectral preconditioning for PDE-constrained optimization in high-dimensional settings.

\begin{figure}[t]
    \centering
    \includegraphics[width=\linewidth]{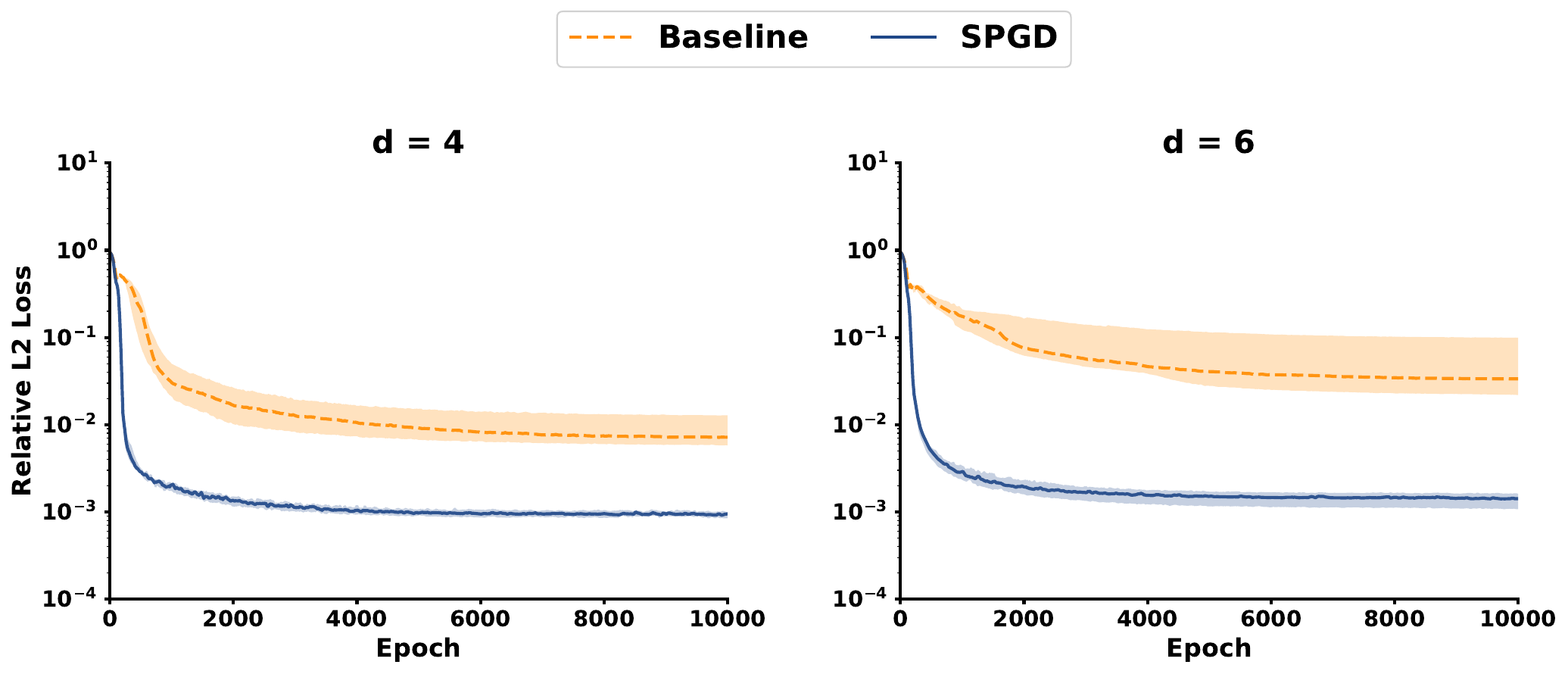}
    \caption{Relative $L^2$ loss versus the number of epochs for the Poisson equation in dimensions $d\in\{4,6\}$. The shaded region in each subplot indicates the interquartile range over 10 independent runs with different random seeds.}
    \label{fig:loss_d4_d6}
\end{figure}

\begin{table}[t]
\centering\small
\caption{Relative $L^2$ loss statistics over $10$ independent runs.}
\label{tab:loss_pde}
\small
\begin{tabular}{c l c c c}
\toprule
$d$ & Optimizer & Median & Best & Worst \\
\midrule
2 & Adam
& $4.31\times10^{-3}$
& $1.89\times10^{-3}$
& $9.28\times10^{-2}$ \\
2 & SPGD
& $\mathbf{9.35\times10^{-4}}$
& $\mathbf{5.01\times10^{-4}}$
& $\mathbf{7.77\times10^{-3}}$ \\
\addlinespace
4 & Adam
& $7.27\times10^{-3}$
& $4.37\times10^{-3}$
& $2.84\times10^{-2}$ \\
4 & SPGD
& $\mathbf{9.30\times10^{-4}}$
& $\mathbf{8.03\times10^{-4}}$
& $\mathbf{1.22\times10^{-3}}$ \\
\addlinespace
6 & Adam
& $3.33\times10^{-2}$
& $7.24\times10^{-3}$
& $1.72\times10^{-1}$ \\
6 & SPGD
& $\mathbf{1.42\times10^{-3}}$
& $\mathbf{8.30\times10^{-4}}$
& $\mathbf{1.76\times10^{-3}}$ \\
\addlinespace
8 & Adam
& $1.12\times10^{-1}$
& $1.67\times10^{-2}$
& $1.62\times10^{-1}$ \\
8 & SPGD
& $\mathbf{3.04\times10^{-3}}$
& $\mathbf{2.38\times10^{-3}}$
& $\mathbf{4.09\times10^{-3}}$ \\
\bottomrule
\end{tabular}
\end{table}

\begin{table}[t]
\centering\small
\caption{Milestone epoch statistics over $10$ independent runs.}
\label{tab:milestone_pde}
\small
\setlength{\tabcolsep}{4pt}
\begin{tabular}{c l c c c c c c }
\toprule
 &  & \multicolumn{3}{c}{Threshold $10^{-1}$}
 & \multicolumn{3}{c}{Threshold $10^{-2}$}
  \\
\cmidrule(lr){3-5}\cmidrule(lr){6-8}
$d$ & Optimizer 
& Median & Best & Worst
& Median & Best & Worst
 \\
\midrule
2 & Adam
& 788  & 460  & 6309
& 1585 & 1078 & 3830
 \\
2 & SPGD
& \textbf{233} & \textbf{156} & \textbf{402}
& \textbf{291} & \textbf{214} & \textbf{614}
 \\
\addlinespace
4 & Adam
& 612  & 366  & 1552
& 2568 & 1585 & 4542
 \\
4 & SPGD
& \textbf{170} & \textbf{146} & \textbf{214}
& \textbf{221} & \textbf{191} & \textbf{326}
 \\
\addlinespace
6 & Adam
& 1501 & 471  & 4948
& 3867 & 3867 & 3867
 \\
6 & SPGD
& \textbf{165} & \textbf{127} & \textbf{193}
& \textbf{300} & \textbf{256} & \textbf{473}
 \\
\addlinespace
8 & Adam
& 2172 & 894  & 5817
& --   & --   & --
 \\
8 & SPGD
& \textbf{158} & \textbf{149} & \textbf{198}
& 485 & 418 & 689
 \\
\bottomrule
\end{tabular}
\end{table}

\subsection{Classification on CIFAR-10}
\label{subsec:public_dataset}
We compare the proposed SPGD with the standard Adam optimizer on the CIFAR-10 dataset using three neural network architectures of increasing complexity. The corresponding model specifications and selected hyperparameters are summarized in~\cref{tab:cifar10_models}. All experiments employ a batch size of $128$, a fixed learning rate of $10^{-2}$, and standard Adam parameters $\beta_1=0.9$, $\beta_2=0.999$, and $\epsilon=10^{-8}$ for both methods. Training is performed for $50$ epochs.

Standard data augmentation is applied, including random horizontal flipping with probability $0.5$ and random cropping with padding size $4$, together with standard CIFAR-10 normalization using mean $[0.491,\,0.482,\,0.447]$ and standard deviation $[0.247,\,0.243,\,0.262]$. 

The configurations reported in~\cref{tab:cifar10_models} illustrate the dependence of the proposed method on network scale. In particular, smaller networks favor weaker damping parameters ($\delta=10^{-5}$) and fewer Lanczos iterations ($k=10$), whereas larger networks require stronger regularization ($\delta=10^{-4}$) to ensure numerical stability of the spectral preconditioner. This behavior is consistent with the increasing ill-conditioning of the Jacobian in deeper and wider architectures.

All experiments are implemented in JAX with deterministic execution enabled to ensure reproducibility.

\begin{table}[t]
\centering
\caption{Network architectures and SPGD hyperparameters for CIFAR-10 experiments.}
\label{tab:cifar10_models}
\small
\begin{tabular}{llc}
\toprule
\textbf{Model} & \textbf{Architecture} & \textbf{SPGD $(\delta, k)$} \\
\midrule
SimpleCNN-8k   & 2 Conv (4, 8) + 2 FC (16, 10)        & $(10^{-5}, 10)$ \\
LeNet-62k      & 2 Conv (6, 16) + 3 FC (120, 84, 10) & $(10^{-5}, 20)$ \\
ResNet20-272k  & 20-layer ResNet + GroupNorm & $(10^{-4}, 10)$ \\
\bottomrule
\end{tabular}
\end{table}

\cref{fig:cifar10_combined} summarizes the comparison. Across the three network architectures (SimpleCNN-8k, LeNet-62k, and ResNet20-272k), SPGD consistently outperforms Adam. The advantage manifests in two respects. First, SPGD converges noticeably faster than Adam: by exploiting curvature information through the preconditioner $J_F^{\top} C J_F$, the method reduces the number of epochs required to reach a given accuracy, as reflected in the training loss and test accuracy curves. Second, SPGD achieves higher final test accuracy than Adam under the same training budget, indicating that the preconditioned updates not only accelerate convergence but also lead to better solutions. These results show that the benefit of SPGD over the first-order baseline holds across architectures of varying capacity and is reflected in both convergence speed and final performance.

\begin{figure}[htbp]
    \centering
    \includegraphics[width=\linewidth]{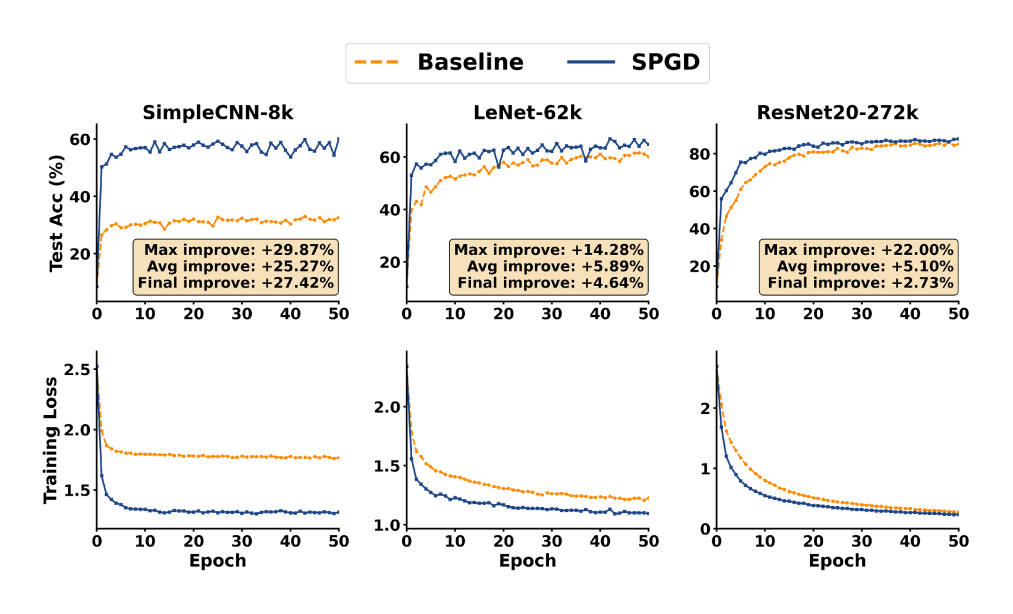}
    \caption{Comparison of Adam (Baseline) and SPGD on three network architectures for CIFAR-10 classification. Top row: test accuracy versus epoch. Bottom row: training loss versus epoch. For each network, improvement statistics (Max/Avg/Final) quantify SPGD's advantage over the baseline. SPGD consistently achieves faster convergence and higher accuracy, with the most pronounced improvements observed in smaller networks.}
    \label{fig:cifar10_combined}
\end{figure}

\section{Conclusions}
\label{sec:conclusions}
This work is motivated by the observation that, in many scientific computing problems formulated as nonlinear least-squares systems, the physical flowâ€”namely, the natural dynamical evolution induced by the governing equationsâ€”often converges to a steady state significantly faster than the standard gradient flow. While gradient flow guarantees monotonic decrease of the residual and global convergence under mild conditions, its convergence can be slow near stationary points due to ill-conditioning of the associated Fisher information or Gauss--Newton matrix. In contrast, physical flow better reflects the intrinsic dynamics of the problem and can exhibit rapid convergence, but it typically lacks robustness and global convergence guarantees.

To reconcile these two perspectives, we introduce an optimization framework that preconditions the gradient descent direction using spectral information from the SVD of the Jacobian. This construction effectively transforms the gradient flow into a \emph{physical-like flow} that preserves the stability and convergence guarantees of gradient-based methods while substantially accelerating convergence. The resulting SVD-preconditioned descent direction is further integrated into the adaptive learning-rate mechanism of Adam, yielding a method that combines problem-dependent geometric information with the robustness of adaptive first-order optimization.

From a theoretical standpoint, we establish local linear convergence results for both classical GD and the proposed SPGD method in the context of nonlinear least-squares problems. In particular, we show that the convergence rate of the preconditioned method is governed by the condition number $\sigma_{\max}/\sigma_{\min}$ of the Jacobian, rather than its square $(\sigma_{\max}/\sigma_{\min})^2$ as in standard gradient descent, leading to a substantial improvement in ill-conditioned regimes. For a modified version of the algorithm, we further provide a global convergence analysis under standard regularity assumptions.

The effectiveness of the proposed approach is demonstrated through numerical experiments on function approximation, PDE-constrained optimization, and image classification on CIFAR-10. Across all examples, replacing the standard gradient direction with the SVD-preconditioned (physical-like) direction consistently accelerates convergence and improves final optimization accuracy, particularly for problems characterized by high intrinsic dimensionality or severely ill-conditioned Jacobians. The CIFAR-10 results further indicate that the method remains robust across varying network scales, with especially pronounced benefits for small to moderately sized models.

Several directions for future research naturally follow from this work. On the theoretical side, it is of interest to further characterize the optimal choice of the preconditioning exponent and to extend the analysis to broader classes of dynamical systems and loss geometries. On the computational side, future efforts will focus on scalable implementations using randomized SVD or Lanczos-type methods, as well as applications to time-dependent, inverse, and data-driven problems in scientific machine learning.

\section*{Acknowledgments}
This research supported by National Institute of General Medical Sciences through grant 1R35GM146894 (ZC and WH).

\bibliographystyle{abbrv}
\bibliography{references}

@article{zhou-2024-convergence_of_adagrad,
  title={On the Convergence of Adaptive Gradient Methods for Nonconvex Optimization},
  author={Zhou, Dongruo and Chen, Jinghui and Cao, Yuan and Yang, Ziyan and Gu, Quanquan},
  journal={Transactions on Machine Learning Research},
  issn={2835-8856},
  year={2024},
}

@article{chang2026fisher,
  title={Fisher-Informed Parameterwise Aggregation for Federated Learning with Heterogeneous Data},
  author={Chang, Zhipeng and He, Ting and Hao, Wenrui},
  journal={arXiv preprint arXiv:2601.13608},
  year={2026}
}

@article{park2021accelerated,
  title={Accelerated additive Schwarz methods for convex optimization with adaptive restart},
  author={Park, Jongho},
  journal={Journal of Scientific Computing},
  volume={89},
  number={3},
  pages={58},
  year={2021},
  publisher={Springer}
}

@article{park2020additive,
  title={Additive Schwarz methods for convex optimization as gradient methods},
  author={Park, Jongho},
  journal={SIAM Journal on Numerical Analysis},
  volume={58},
  number={3},
  pages={1495--1530},
  year={2020},
  publisher={SIAM}
}

@inproceedings{muller2023achieving,
  title={Achieving high accuracy with PINNs via energy natural gradient descent},
  author={M{\"u}ller, Johannes and Zeinhofer, Marius},
  booktitle={International Conference on Machine Learning},
  pages={25471--25485},
  year={2023},
  organization={PMLR}
}

@article{hao2025efficient,
  title={An Efficient Quasi-Newton Method with Tensor Product Implementation for Solving Quasi-Linear Elliptic Equations and Systems},
  author={Hao, Wenrui and Lee, Sun and Zhang, Xiangxiong},
  journal={Journal of Scientific Computing},
  volume={103},
  number={3},
  pages={89},
  year={2025},
  publisher={Springer}
}

@article{hao2024gauss,
  title={Gauss Newton method for solving variational problems of PDEs with neural network discretizations},
  author={Hao, Wenrui and Hong, Qingguo and Jin, Xianlin},
  journal={Journal of Scientific Computing},
  volume={100},
  number={1},
  pages={17},
  year={2024},
  publisher={Springer}
}

@article{hao2021gradient,
  title={A gradient descent method for solving a system of nonlinear equations},
  author={Hao, Wenrui},
  journal={Applied Mathematics Letters},
  volume={112},
  pages={106739},
  year={2021},
  publisher={Elsevier}
}

@article{ruder2016overview,
  title={An overview of gradient descent optimization algorithms},
  author={Ruder, Sebastian},
  journal={arXiv preprint arXiv:1609.04747},
  year={2016}
}

@misc{yang-2016-unified_convergence_analysis_stochastic,
  title={Unified Convergence Analysis of Stochastic Momentum Methods for Convex and Non-convex Optimization},
  author={Yang, Tianbao and Lin, Qihang and Li, Zhe},
  year={2016},
  eprint={1604.03257},
  archivePrefix={arXiv},
  primaryClass={math.OC},
}

@article{chen2022randomized,
  title={Randomized Newton’s method for solving differential equations based on the neural network discretization},
  author={Chen, Qipin and Hao, Wenrui},
  journal={Journal of Scientific Computing},
  volume={92},
  number={2},
  pages={49},
  year={2022},
  publisher={Springer}
}

@article{siegel2021greedy,
  title={Greedy training algorithms for neural networks and applications to PDEs},
  author={Siegel, Jonathan W and Hong, Qingguo and Jin, Xianlin and Hao, Wenrui and Xu, Jinchao},
  journal={arXiv preprint arXiv:2107.04466},
  year={2021}
}

@inproceedings{reddi-2018-AMSGrad,
  title={On the Convergence of Adam and Beyond},
  author={Reddi, Sashank J and Kale, Satyen and Kumar, Sanjiv},
  booktitle={International Conference on Learning Representations},
  year={2018}
}

@misc{kingma-2017-adam,
  title={Adam: A Method for Stochastic Optimization},
  author={Kingma, Diederik P. and Ba, Jimmy},
  year={2017},
  eprint={1412.6980},
  archivePrefix={arXiv},
  primaryClass={cs.LG},
}

@article{duchi-2011-adagrad,
  title={Adaptive subgradient methods for online learning and stochastic optimization},
  author={Duchi, John and Hazan, Elad and Singer, Yoram},
  journal={Journal of Machine Learning Research},
  volume={12},
  number={7},
  year={2011}
}

@inproceedings{martens-2015-KFAC,
  title={Optimizing neural networks with kronecker-factored approximate curvature},
  author={Martens, James and Grosse, Roger},
  booktitle={International conference on machine learning},
  pages={2408--2417},
  year={2015},
  organization={PMLR}
}

@inproceedings{gupta-2018-shampoo,
  title={Shampoo: Preconditioned stochastic tensor optimization},
  author={Gupta, Vineet and Koren, Tomer and Singer, Yoram},
  booktitle={International Conference on Machine Learning},
  pages={1842--1850},
  year={2018},
  organization={PMLR}
}

@article{liu-1989-LBFGS,
  title={On the limited memory BFGS method for large scale optimization},
  author={Liu, Dong C and Nocedal, Jorge},
  journal={Mathematical Programming},
  volume={45},
  number={1},
  pages={503--528},
  year={1989},
  publisher={Springer}
}

@book{nocedal-2006-numerical_optimization,
  title={Numerical Optimization},
  author={Nocedal, Jorge and Wright, Stephen J},
  year={2006},
  publisher={Springer}
}

@article{amari-1998-natural_gradient,
  title={Natural gradient works efficiently in learning},
  author={Amari, Shun-Ichi},
  journal={Neural Computation},
  volume={10},
  number={2},
  pages={251--276},
  year={1998},
  publisher={MIT Press}
}

@article{polyak-1963-gradient,
title = {Gradient methods for the minimisation of functionals},
journal = {USSR Computational Mathematics and Mathematical Physics},
volume = {3},
number = {4},
pages = {864--878},
year = {1963},
issn = {0041-5553},
doi = {https://doi.org/10.1016/0041-5553(63)90382-3},
url = {https://www.sciencedirect.com/science/article/pii/0041555363903823},
author = {Polyak, B.T.}
}

@misc{lojasiewicz-1963-topological,
 author = {{\L}ojasiewicz, Stanis{\l}aw},
 title = {A topological property of real analytic subsets},
 year = {1963},
 language = {French},
 howpublished = {Equ. {Derivees} partielles, {Paris} 1962, {Colloques} internat. {Centre} nat. {Rech}. sci. 117, 87-89 (1963).},
 zbMATH = {3371284},
 Zbl = {0234.57007}
}

@article{xu-2019-frequency-principle,
  title={Frequency principle: Fourier analysis sheds light on deep neural networks},
  author={Xu, Zhi-Qin John and Zhang, Yaoyu and Luo, Tao and Xiao, Yanyang and Ma, Zheng},
  journal={arXiv preprint arXiv:1901.06523},
  year={2019}
}

@Inbook{Lee-2012,
author="Lee, John M.",
title="Submanifolds",
bookTitle="Introduction to Smooth Manifolds",
year="2012",
publisher="Springer New York",
address="New York, NY",
pages="98--124",
isbn="978-1-4419-9982-5",
doi="10.1007/978-1-4419-9982-5_5",
url="https://doi.org/10.1007/978-1-4419-9982-5_5"
}

@article{feng-2022-rank_diminishinig,
  title={Rank diminishing in deep neural networks},
  author={Feng, Ruili and Zheng, Kecheng and Huang, Yukun and Zhao, Deli and Jordan, Michael and Zha, Zheng-Jun},
  journal={Advances in Neural Information Processing Systems},
  volume={35},
  pages={33054--33065},
  year={2022}
}

@InProceedings{karakida-2019-universal_statistics_of_FIM,
  title = 	 {Universal Statistics of Fisher Information in Deep Neural Networks: Mean Field Approach},
  author =       {Karakida, Ryo and Akaho, Shotaro and Amari, Shun-ichi},
  booktitle = 	 {Proceedings of the Twenty-Second International Conference on Artificial Intelligence and Statistics},
  pages = 	 {1032--1041},
  year = 	 {2019},
  editor = 	 {Chaudhuri, Kamalika and Sugiyama, Masashi},
  volume = 	 {89},
  series = 	 {Proceedings of Machine Learning Research},
  month = 	 {16--18 Apr},
  publisher =    {PMLR},
  url = 	 {https://proceedings.mlr.press/v89/karakida19a.html}
}

@book{grossmann-2007-numerical_treatment,
  title={Numerical Treatment of Partial Differential Equations: Translated and revised by Martin Stynes},
  author={Grossmann, Christian and Roos, Hans-G{\"o}rg and Stynes, Martin},
  publisher={Springer},
  year={2007}
}

@book{golub-2013-matrix_computations,
  title={Matrix Computations},
  author={Golub, Gene H. and Van Loan, Charles F.},
  edition={4},
  year={2013},
  publisher={Johns Hopkins University Press}
}

\end{document}


\begin{center}
{\LARGE\sc SVD-Preconditioned Gradient Descent Method for Solving Nonlinear Least Squares Problems}\\[8pt]
{\large Supplementary Material}\\[12pt]
{\normalsize Zhipeng Chang\quad Wenrui Hao$^*$\quad Nian Liu}\\[4pt]
{\small Department of Mathematics, Penn State University, University Park, PA 16802, USA}
\end{center}
\vspace{0.3in}

\appendix
\section{Lemmas for Section 2} 
\label{sec:proof of lemmas}

\begin{lemma}\label{lem:lower bound for Hessian along normal direction}
Let $\theta_0$ be any point in $\mathcal{M}$. Then there exists a neighborhood $V \subset U$ of $\theta_0$, together with a projection map $\pi: V \to \mathcal{M}$, such that: 
\begin{itemize}
    \item For any constant $\delta > 0$, the entire line segment connecting $\theta$ and $\pi(\theta)$,
    \begin{equation*}
        [\pi(\theta), \theta] := \{ (1-t)\pi(\theta) + t\theta : t \in [0,1]\},
    \end{equation*}
    lies within $V$ for all $\theta \in V$ and satisfies $\|\theta - \pi(\theta)\|_2 < \delta$.
    
    \item The image $\pi(V)$ is an open neighborhood of $\theta_0$ in $\mathcal{M}$, which can be chosen arbitrarily small around $\theta_0$.
    
    \item For every $\theta \in V$, the Hessian $\nabla^2 f(\theta)$ is uniformly positive definite on the normal space $N_{\pi(\theta)}\mathcal{M}$. More precisely, let $\sigma_{\min}(\theta_0)$ denote the smallest nonzero singular value of the Jacobian $J_F(\theta_0)$. Then there exists a constant $\mu = \sigma_{\min}^2(\theta_0)/4 > 0$ such that
    \begin{equation*}
        \langle \nabla^2 f(\theta)\xi, \xi \rangle \ge 2\mu \|\xi\|_2^2, 
        \qquad \forall \xi \in N_{\pi(\theta)}\mathcal{M}, \quad\forall \theta \in V.
    \end{equation*}
    Equivalently, the restriction of $\nabla^2 f(\theta)$ to the normal space $N_{\pi(\theta)}\mathcal{M}$ has a uniformly positive lower bound on its spectrum for all $\theta\in V$.
\end{itemize}
\end{lemma}

\begin{proof}[Proof]
We pick a precompact neighborhood $S$ of $\theta_0$ in $\mathcal{M}$. 
For any $p \in S$, consider the normal exponential map
\begin{equation}\label{equ:exponential on normal space}
        \begin{split}
            \exp^{\perp}_p: U_p\times N_{p}\mathcal{M}&\longrightarrow\mathbb{R}^m,
            \\ (q,v)&\mapsto q+v,
        \end{split}
    \end{equation}
where $U_p \subset \mathcal{M}$ is an open neighborhood of $p$ in $\mathcal{M}$. 
At the point $(p,0)$, the differential of~\cref{equ:exponential on normal space} given by
\begin{equation*}
    D\exp_p^{\perp}(p,0)(\xi,\eta) = \xi + \eta,
    \qquad\forall \xi \in T_p\mathcal{M},\quad\forall\eta \in N_p\mathcal{M},
\end{equation*} is a linear isomorphism. Therefore, the inverse function theorem ensures that, after possibly shrinking $U_p$, there exists a radius $r_p>0$ such that
\begin{equation}\label{equ:exponential is local diffeomorphism}
    \exp_p^{\perp} : U_p \times W_p \to \exp_p^{\perp}(U_p\times W_p)
\end{equation}
is a diffeomorphism onto its image, where $W_p:=\{v\in N_p\mathcal{M}:\|v\|_2<r_p\}.$

Let $K := \overline{S}$ be the closure of $S$ in $\mathcal{M}$. 
Since $S$ is precompact, $K$ is compact. 
For each $p \in K$, choose the corresponding neighborhoods $U_p$ and $W_p$ in~\cref{equ:exponential is local diffeomorphism}. The family $\{U_p\}_{p \in K}$ forms an open cover of $K$. The Lebesgue number of $\{U_p\}_{p\in K}$ is denoted by $\delta_{K}$. By compactness, there exist finitely many points $p_1,\dots,p_k \in K$ 
such that $\{U_{p_i}\}_{i=1}^k$ covers $K$. 
Define
\begin{equation}\label{equ:definition of radius}
    r := \min\left\{ \min_{1 \le i \le k}r_{p_i},\delta,\frac{\delta_K}{2} \right\} > 0.
\end{equation}

We define
\begin{equation*}
    N_S^r\mathcal{M}=\{(p,v):p\in S, v\in N_p\mathcal{M},\|v\|_2<r\}.
\end{equation*}
We now prove that the map
\begin{equation}\label{equ:definition of exp perp on S}
    \begin{split}
        \exp^{\perp}:S\times N_S^r\mathcal{M}&\longrightarrow\mathbb{R}^m\\ (p, v)&\mapsto\exp_p^{\perp}(v)
    \end{split}
\end{equation} is injective. The map $\exp^{\perp}$ is clearly well defined. Suppose there exist two pairs $(p_i,v_i)$ and $(p_j,v_j)$, with $p_i,p_j\in S, v_i\in {N}_{p_i}\mathcal{M}, v_j\in N_{p_j}\mathcal{M}$ such that $$p_i+v_i=p_j+v_j.$$ Then, this implies
\begin{equation*}
    p_i-p_j=v_j-v_i,
\end{equation*}and hence
\begin{equation*}
    \|p_i-p_j\|_2=\|v_i-v_j\|_2\leq 2r\leq \delta_K.
\end{equation*}By the definition of the Lebesgue number $\delta_K$, both $p_i$ and $p_j$ must lie in some common neighborhood $U_{p^*}$. Since the restriction $\exp^{\perp}|_{U_{p^*}}$ is a diffeomorphism onto its image, the equality $$\exp^{\perp}|_{U_{p^*}}(p_i,v_i)=\exp^{\perp}|_{U_{p^*}}(p_j,v_j)$$ yields $p_i=p_j, v_i=v_j$. Therefore, $\exp^{\perp}$ is locally a diffeomorphism and globally injective; in other words, $\exp^{\perp}$ is a diffeomorphism onto its image $\mathcal{T}_S=\exp^{\perp}(N_S^r\mathcal{M})$. We claim that $\mathcal{T}_S$ is exactly the neighborhood $V$ that we are looking for.

We define the projection $\pi:\mathcal{T}_S\to\mathcal{M}$ as $\pi:=\pi_S\circ(\exp^{\perp})^{-1}$. Combining~\cref{equ:definition of radius} and~\cref{equ:definition of exp perp on S} yields that, for every $\theta \in \mathcal{T}_S$,
\[
    \theta - \pi(\theta) \in N_{\pi(\theta)}\mathcal{M},
    \qquad
    \|\theta - \pi(\theta)\|_2 < r\leq\delta.
\]
Moreover, for any $t\in[0,1],$ we have
\begin{equation*}
    (1-t)\pi(\theta)+t\cdot\theta=\pi(\theta)+t(\theta-\pi(\theta))=\exp_{\pi(\theta)}^{\perp}(t(\theta-\pi(\theta))).
\end{equation*}
In other words, the entire line segment $[\pi(\theta),\theta]$ lies within $\mathcal{T}_S.$ Besides, notice that the neighborhood $S=\pi(\mathcal{T}_S)$ is a prescribed neighborhood; it can be chosen arbitrarily small around $\theta_0$.

Finally, shrinking $r$ if necessary, and by the uniform continuity of $\nabla^2 f(\theta)$ on $\mathcal{T}_S$, for any $\theta\in\mathcal{T}_S$, we have
\begin{equation}\label{equ:closeness}
    \|\nabla^2f(\theta)-\nabla^2f(\pi(\theta))\|_2\leq\frac{\sigma^2_{\min}(\theta_0)}{4}.
\end{equation}
Moreover, notice that
\begin{equation*}\label{equ:the qudratic form on manifold}
    \begin{split}
    \Phi(\cdot,\cdot):S\mathcal{M}&\longrightarrow\mathbb{R},
        \\ 
        (p,\xi)&\mapsto\langle\nabla^2f(p)\xi,\xi\rangle
    \end{split}
\end{equation*}
is continuous, where the unit normal bundle over $\mathcal{M}$ is defined as
\begin{equation*}
    S\mathcal{M}:=\{(p,\xi):p\in\mathcal{M},\xi\in N_p\mathcal{M},\|\xi\|_2=1\}.
\end{equation*}
Since 
\begin{equation*}
\Phi(\theta_0,\xi)=\langle\nabla^2f(\theta_0)\xi,\xi\rangle\geq \sigma_{\min}^2(\theta_0),\qquad\forall\xi\in N_{\theta_0}\mathcal{M}.
\end{equation*}Shrinking $S$ around $\theta_0$ if necessary, by the continuity of~\cref{equ:the qudratic form on manifold}, we may assume that
\begin{equation*}
    \Phi(p,\xi)\geq\frac{3}{4}\sigma_{\min}^2(\theta_0),\qquad\forall p\in S,\,\forall\xi\in N_{p}\mathcal{M}.
\end{equation*}
This implies that
\begin{equation}\label{equ:lower bound around theta_0}
\langle\nabla^2f(p)\xi,\xi\rangle\geq\frac{3}{4}\sigma_{\min}^2(\theta_0)\|\xi\|^2_2,\qquad\forall p\in S,\,\forall\xi\in N_p\mathcal{M}.
\end{equation}
Let $\mu := \sigma_{\min}^2(\theta_0)/4 > 0$. Using \eqref{equ:closeness} and \eqref{equ:lower bound around theta_0}, for any $\theta\in\mathcal{T}_S$ and any $\xi\in N_{\pi(\theta)}\mathcal M$,
\begin{equation*}
    \begin{split}
\langle\nabla^2f(\theta)\xi,\xi\rangle
&= \langle(\nabla^2f(\theta)-\nabla^2f(\pi(\theta)))\xi,\xi\rangle
  + \langle\nabla^2f(\pi(\theta))\xi,\xi\rangle
\\
&\geq -\|\nabla^2f(\theta)-\nabla^2f(\pi(\theta))\|_2\|\xi\|_2^2
    + \frac{3}{4}\sigma_{\min}^2(\theta_0)\|\xi\|_2^2
\\
&\geq 2\mu \|\xi\|_2^2.
    \end{split}
\end{equation*}
That is, the restriction of $\nabla^2 f(\theta)$ to $N_{\pi(\theta)}\mathcal{M}$ 
admits the uniform spectral lower bound $2\mu$ on $V:=\mathcal{T}_S$.

\end{proof}

\begin{lemma}\label{lem:PL inequality for the objective function}
Let $V \subset U$ and the projection $\pi:V\to\mathcal{M}$ be as given in~\cref{lem:lower bound for Hessian along normal direction}. 
Then there exists a constant $\mu = c_0/2 > 0$ such that
\begin{equation*}
    \|\nabla f(\theta)\|_2^2 \geq 2\mu(f(\theta)-f(\pi(\theta))),
    \qquad \forall\,\theta\in V.
\end{equation*}
In particular, since $f(\pi(\theta))=0$ for $\theta\in\mathcal{M}$, the inequality reduces to
\begin{equation*}
    \|\nabla f(\theta)\|_2^2 \geq 2\mu f(\theta),
    \qquad \forall\theta\in V.
\end{equation*}
\end{lemma}

The above inequality is known as the Polyak--\L{}ojasiewicz (PL) inequality (see~\cite{polyak-1963-gradient} for the case where the minimal function value equals zero, and~\cite{lojasiewicz-1963-topological}), which plays a central role in establishing the linear convergence of the gradient descent method for $\mu$-strongly convex functions. However, the proof of~\cref{lem:PL inequality for the objective function} is nontrivial, since our objective function is not $\mu$-strongly convex. In particular, observe that
\begin{equation*}
    \nabla^2 f(\theta)
    = J_F^{\top}(\theta) J_F(\theta)
    + \sum_{i=1}^n F_i(\theta) \nabla^2 F_i(\theta).
\end{equation*}
Even at $\theta = \theta^*$, the Hessian simplifies to
\[
    \nabla^2 f(\theta^*) = J_F^{\top}(\theta^*) J_F(\theta^*),
\]
which possesses zero eigenvalues whenever $r < m$. Consequently, there is no guarantee that the smallest eigenvalue of $\nabla^2 f(\theta)$ admits a uniform positive lower bound. In other words, the objective function $f$ does not satisfy the $\mu$-strong convexity condition.

\begin{proof}[Proof]
By~\cref{lem:lower bound for Hessian along normal direction}, there exists an open set 
$V \subset U$ and a projection $\pi:V\to \mathcal{M}$ such that $\pi(\theta)=\theta$ for all 
$\theta\in\mathcal{M}$, and every $\theta\in V$ can be uniquely written as
\begin{equation*}
    \theta = \pi(\theta) + n(\theta),
\end{equation*}
where $n(\theta)\in N_{\pi(\theta)}\mathcal{M}$. 
For each $\theta\in V$, define the straight line connecting $\pi(\theta)$ and $\theta$ by
\begin{equation*}
\begin{split}
    \gamma_\theta:[0,1]&\longrightarrow V,
    \\
    t&\mapsto \pi(\theta) + t\cdot n(\theta).
\end{split}
\end{equation*}
We define the one-dimensional restriction
\[
    \phi(t) := f(\gamma_\theta(t)), \qquad t\in[0,1].
\]
Then $\phi^{\prime\prime}(t)$ can be expressed as
\begin{equation}\label{eq:phi-second}
    \phi''(t) = \langle \nabla^2 f(\gamma_\theta(t))\,n(\theta),\,n(\theta) \rangle.
\end{equation}
By~\cref{lem:lower bound for Hessian along normal direction}, the restriction of $\nabla^2 f(\theta)$ to $N_{\pi(\theta)}\mathcal{M}$ 
admits the uniform spectral lower bound $2\mu$, that is,
\begin{equation}\label{eq:hessian-lb}
    \langle \nabla^2 f(\theta)\xi,\xi\rangle 
    \ge 2\mu\|\xi\|_2^2,
    \qquad \forall\theta\in V,\forall\xi\in N_{\pi(\theta)}\mathcal{M}.
\end{equation}
Hence 
\begin{equation}\label{equ:lower bound of the second derivative of phi}
    \phi''(t)\geq 2\mu\|n(\theta)\|_2^2,\quad \forall t\in[0,1].
\end{equation}

By the chain rule and the fundamental theorem of calculus,
\begin{equation}\label{equ:f difference}
    f(\theta)-f(\pi(\theta)) = \phi(1)-\phi(0) 
    = \int_0^1 \phi'(s)\,\mathrm{d}s.
\end{equation}
Since $\pi(\theta)\in\mathcal{M}$ is a minimizer of $f$, we have $\nabla f(\pi(\theta))=0$, 
and hence
\begin{equation}\label{equ:phi prime zero}
    \phi'(0) = \langle \nabla f(\pi(\theta)),\,n(\theta) \rangle = 0.
\end{equation}
Combining~\cref{equ:f difference} and~\cref{equ:phi prime zero} gives
\begin{equation}\label{equ:f diff expanded}
\begin{split}
     f(\theta)-f(\pi(\theta))
     &= \int_0^1 \phi'(s)\,\mathrm{d}s
     = \int_0^1 \left(\int_0^s \phi''(t)\mathrm{d}t\right)\mathrm{d}s
     \\[3pt]
     &= \int_0^1 (1-t)\phi''(t)\mathrm{d}t.
\end{split}
\end{equation}

Similarly, differentiating along $\gamma_\theta$ yields
\begin{equation}\label{equ:grad diff}
    \nabla f(\theta) 
    = \int_0^1 \nabla^2 f(\gamma_\theta(t))\,n(\theta)\,\mathrm{d}t.
\end{equation}
Applying the Cauchy-Schwarz inequality gives
\begin{equation}\label{equ:grad lower bound}
\begin{split}
    \|\nabla f(\theta)\|_2^2
    &\ge 
    \frac{\left(\displaystyle\int_0^1 
        \langle \nabla^2 f(\gamma_\theta(t))n(\theta),n(\theta) \rangle
        \mathrm{d}t
    \right)^2}
    {\|n(\theta)\|_2^2}
    \\[3pt]
    &= \frac{\left(\displaystyle\int_0^1 \phi''(t)\mathrm{d}t\right)^2}{\|n(\theta)\|_2^2}.
\end{split}
\end{equation}
Since $\phi''(t)\ge0$, combining \cref{equ:f diff expanded} and \cref{equ:grad lower bound} yields
\begin{equation*}
\begin{split}
    \|\nabla f(\theta)\|_2^2
    &\ge
    \frac{\left(\displaystyle\int_0^1 \phi''(t)\mathrm{d}t\right)
          \left(\displaystyle\int_0^1 (1-t)\phi''(t)\,\mathrm{d}t\right)}
          {\|n(\theta)\|_2^2}
    \\[3pt]
    &\ge 2\mu(f(\theta)-f(\pi(\theta))
    \\[3pt]
    &=2\mu f(\theta)
\end{split}
\end{equation*}
This completes the proof.
\end{proof}

\begin{lemma}\label{lem:F approx normal}
    Given any point $\theta_0\in\mathcal{M}$ and the projection map $\pi:V\to\mathcal{M}$ defined in~\cref{lem:lower bound for Hessian along normal direction}, there exists a neighborhood $W\subset V$, such that for any $\theta\in W$, we have
    \begin{equation*}
        \|F(\theta)\|_2\geq \frac{\sigma_{\min}(J_F(\theta_0))}{2}\|n(\theta)\|_2,
    \end{equation*}
    where $n(\theta):= \theta - \pi(\theta)\in N_{\pi(\theta)}\mathcal{M}$ and $\sigma_{\min}(J_F(\theta_0))$ denotes the smallest singular value of the Jacobian matrix $J_F(\theta_0).$
\end{lemma}

\begin{proof}[Proof]
    Since $F$ is $C^2$, there exists a precompact neighborhood $W \subset V$ of $\theta_0$ and a constant $C_F > 0$ such that for all $\theta \in w$, the Taylor expansion satisfies
    \begin{equation*}
        \|F(\theta) - F(\pi(\theta)) - J_F(\pi(\theta))(\theta-\pi(\theta))\|_2 \leq C_F \|\theta-\pi(\theta)\|_2^2.
    \end{equation*}
    Using the triangle inequality and the fact that $F(\pi(\theta))=0$, we have
    \begin{equation*}
        \|F(\theta)\|_2 \geq \|J_F(\pi(\theta))(\theta-\pi(\theta))\|_2 - C_F\|\theta-\pi(\theta)\|_2^2.
    \end{equation*}
    Note that $n(\theta) := \theta-\pi(\theta)$ lies in the normal space $N_{\pi(\theta)}\mathcal{M} = \mathrm{range}(J_{F}^{\top}(\pi(\theta)))$. Therefore,
    \begin{equation*}
        \|J_F(\pi(\theta))n(\theta)\|_2 \geq \sigma_{\min}(J_F(\pi(\theta)))\|n(\theta)\|_2.
    \end{equation*}
   By further restricting the neighborhood $W$, we ensure that for any $\theta \in W$:
\begin{itemize}
    \item the continuity of $\sigma_{\min}$ implies
    \[ \sigma_{\min}(J_F(\pi(\theta))) \geq \frac{3}{4}\sigma_{\min}(J_F(\theta_0)); \]
    \item the distance to the manifold satisfies
    \[ \|n(\theta)\|_2 \le \frac{\sigma_{\min}(J_F(\theta_0))}{4C_F}. \]
\end{itemize}
    Combining these estimates, for any $\theta \in W$, we obtain
    \begin{equation*}
    \begin{split}
         \|F(\theta)\|_2 &\geq \left(\frac{3}{4}\sigma_{\min}(J_F(\theta_0)) - \frac{1}{4}\sigma_{\min}(J_F(\theta_0))\right)\|n(\theta)\|_2
         \\
         &= \frac{\sigma_{\min}(J_F(\theta_0))}{2}\|n(\theta)\|_2,
    \end{split}
    \end{equation*}
    which completes the proof.
\end{proof}

\begin{lemma}\label{lem:null_space_control}
Denote by $\mathcal{R}(\theta) = \mathrm{range}(J_F(\theta))$ the column space of the Jacobian matrix $J_F(\theta)$, and $P_{\mathcal{R}(\theta)^\perp}$ be the orthogonal projection onto its orthogonal complement. Given any point $\theta_0\in \mathcal{M}$ and the projection map $\pi:W\to\mathcal{M}$ defined in~\cref{lem:F approx normal}, there exists a constant $C_{\perp} > 0$ such that for any $\theta \in W$,
\begin{equation}
    \|P_{\mathcal{R}(\theta)^\perp} F(\theta)\|_2 \le C_{\perp} \|F(\theta)\|_2^2.
\end{equation}
\end{lemma}

\begin{proof}[Proof]
   We expand $F(\pi(\theta))$ around the current point $\theta$ using Taylor's theorem:
    \begin{equation*}
        0 = F(\pi(\theta)) = F(\theta) + J_F(\theta)(\pi(\theta) - \theta) + \mathcal{Q},
    \end{equation*}
    where the remainder term satisfies $\|\mathcal{Q}\|_2 \leq \dfrac{L_{J_F}}{2} \|\pi(\theta) - \theta\|_2^2$.
    
    Recalling the definition $n(\theta) = \theta - \pi(\theta)$, we thus obtain
    \begin{equation}\label{equ:F around projection}
        F(\theta) = J_F(\theta) n(\theta) - \mathcal{Q}.
    \end{equation}
    Applying the orthogonal projection operator $P_{\mathcal{R}(\theta)^\perp}$ to both sides of~\cref{equ:F around projection} yields
    \begin{equation*}
        P_{\mathcal{R}(\theta)^\perp} F(\theta) = \underbrace{P_{\mathcal{R}(\theta)^\perp} [J_F(\theta) n(\theta)]}_{=0} - P_{\mathcal{R}(\theta)^\perp} \mathcal{Q} = - P_{\mathcal{R}(\theta)^\perp} \mathcal{Q}.
    \end{equation*}
    Taking norms and using the property that $\|P_{\mathcal{R}(\theta)^\perp}\|_2 \le 1$ gives
    \begin{equation}\label{equ:norm of projection for the residual}
        \|P_{\mathcal{R}(\theta)^\perp} F(\theta)\|_2 \le \|\mathcal{Q}\|_2 \le \frac{L_{J_F}}{2} \|n(\theta)\|_2^2.
    \end{equation}
    Finally, substituting the upper bound for $\|n(\theta)\|_2$ into~\cref{equ:norm of projection for the residual} yields:
    \begin{equation*}
        \|P_{\mathcal{R}(\theta)^\perp} F(\theta)\|_2 \le \frac{L_{J_F}}{2} \left( \frac{2}{\sigma_{\min}(J_F(\theta_0))} \|F(\theta)\|_2 \right)^2.
    \end{equation*}
    Setting $C_{\perp} := \dfrac{2 L_{J_F}}{\sigma_{\min}^2(J_F(\theta_0))}$ completes the proof.
\end{proof}

\section{Auxiliary lemmas for Section 3}
\label{subsubsec:auxiliary lemmas}
In this subsection, we will collect several lemmas from \cite{zhou-2024-convergence_of_adagrad}, whose proofs are given in~\cref{sec:proof of lemmas} for the sake of completeness. These technical results will be used repeatedly in the proof of~\cref{thm:convergence of the modified algorithm}. For notational convenience, we define
\begin{equation}\label{equ:construction of At}
    A_t=\alpha\widehat{D}_t^{-\frac{1}{2}},
\end{equation}a quantity that will be appear frequently in the subsequent analysis.

The operator norms of the diagonal matrices $\big\{A_t^{-1/2}\big\}_{t\ge 1}$ are uniformly bounded. Specifically,

\begin{lemma}\label{lem:2 norm for At}
     The sequence of diagonal matrices $\big\{A_t^{-1/2}\big\}_{t\geq 1}$ is uniformly bounded. In particular,
     \begin{equation*}
         \Big\|A_t^{-\frac{1}{2}}\Big\|_2\leq \sqrt[4]{\frac{C_{\infty}^2M^2+\epsilon}{\alpha^2}},
     \end{equation*}
where $M$ is the constant in Assumption~\ref{assump:bounded gradient}, 
$\alpha$ is the stepsize used in~\cref{alg:svd-adam_algorithm}, 
and $\epsilon$ is the parameter introduced in~\cref{equ:construction of widehat V}.
\end{lemma}

\begin{proof}[Proof]
Recall that by definition~\cref{equ:construction of At} of $A_t$, we have
\begin{equation*}
A_t^{-\frac{1}{2}}=\alpha^{-\frac{1}{2}}\widehat D_t^{\frac{1}{4}},
\end{equation*}
where $\widehat{D}_t$ denotes the preconditioner defined via the modified second moment in~\cref{equ:construction of widehat V}. Since $\widehat D_t^{{1}/{4}}$ is diagonal with diagonal entries $\left(\widehat{v}_{t,i}+\epsilon\right)^{1/4}$, and the spectral norm of a positive diagonal matrix equals its largest diagonal entry. It follows that
\begin{equation*}
\big\|A_t^{-\frac{1}{2}}\big\|_2
=
\alpha^{-\frac{1}{2}}
\big\|\widehat D_t^{\frac{1}{4}}\big\|_2
=
\alpha^{-\frac{1}{2}}
\max_{i}\left(\widehat v_{t,i}+\epsilon\right)^{\frac{1}{4}}.
\end{equation*}
Using the uniform bound $\|\lambda_t\|_{\infty}\leq C_{\infty}M$ in~\cref{equ:lambdat is bounded} together with the construction of $\widehat{v}_t$, we have  $\widehat v_{t,i}\le C_{\infty}^2M^2$ for all $t,i$. Consequently,
\begin{equation*}
\big\|A_t^{-\frac{1}{2}}\big\|_2
\leq
\alpha^{-\frac{1}{2}}
\left(C_{\infty}^2M^2+\epsilon\right)^{\frac{1}{4}}
=
\sqrt[4]{\frac{C_{\infty}^2M^2+\epsilon}{\alpha^2}},
\end{equation*}
which establishes the desired uniform bound.
\end{proof}

The following descent lemma follows from the $L$-smoothness of $f$ in Assumption~\ref{assump:L-smoothness}.

\begin{lemma}\label{lem:descent lemma}
For any $x,y\in\mathbb{R}^m,$ we have 
    \begin{equation}
        f(y)\leq f(x)+\langle \nabla f(x), y-x\rangle +\frac{L}{2}\|y-x\|^2_2,
    \end{equation}
    where $L$ is the Lipschitz constant of $\nabla f$.
\end{lemma}

\begin{proof}[Proof]
    We define $\phi(t):[0,1]\to\mathbb{R}$ as
    \begin{equation*}\label{equ:definition of phi(t)}
        \phi(t)= f((1-t)x+ty).
    \end{equation*}
    By the fundamental theorem of calculus, we obtain
    \begin{align*}
        \phi(1)-\phi(0)&=\int_{0}^1\phi^{\prime}(t)\mathrm{d}t\\&=\int_{0}^1\langle\nabla f((1-t)x+ty),y-x\rangle\mathrm{d}t\\&=\int_0^1\langle\nabla f(x), y-x\rangle\mathrm{d}t+\int_0^1\langle \nabla f((1-t)x+ty)-\nabla f(x), y-x\rangle\mathrm{d}t\\&\leq \langle\nabla f(x), y-x\rangle+ \int_{0}^1\|\nabla f((1-t)x+ty)-\nabla f(x)\|_2\|y-x\|_2 \mathrm{d}t\\&\leq\langle\nabla f(x), y-x\rangle+\int_0^1tL\|y-x\|_2^2\mathrm{d}t\\&=\langle\nabla f(x), y-x\rangle + \frac{L}{2}\|y-x\|_2^2.
    \end{align*}
Notice the fact that $\phi(0)=x,\phi(1)=y$, the above reasoning completes the proof.
\end{proof}

For the sake of completeness, we recall several auxiliary lemmas from~\cite{zhou-2024-convergence_of_adagrad}, which will be used in our subsequent analysis.

\begin{lemma}\label{lem:expression for the difference of z_t}
     We have the expression of $z_{t+1}-z_t$:
        \begin{equation*}\label{equ:the difference of z_t}
            z_{t+1}-z_t=\begin{cases}
                \dfrac{\beta_1}{1-\beta_1}(A_{t-1}-A_t)m_{t-1}-A_t\lambda_t, &t\geq 2,\vspace{0.3cm}\\ -A_1\lambda_1, &t=1.
            \end{cases}
        \end{equation*}
\end{lemma}

\begin{proof}[Proof]
For $t=1,$ notice that
\begin{equation}\label{equ:calculation of z_2 and z_1}
\begin{split}
    z_1 &= \theta_1+\frac{\beta_1}{1-\beta_1}(\theta_1-\theta_0)=\theta_1,\\
    z_2 &= \theta_2+\frac{\beta_1}{1-\beta_1}(\theta_2-\theta_1).
\end{split}
\end{equation}
Therefore, we have
\begin{equation*}
    z_2-z_1=\frac{1}{1-\beta_1}(\theta_2-\theta_1)=-\frac{1}{1-\beta_1}A_1m_1=-A_1\lambda_1.
\end{equation*}
For $t\geq 2$, it follows that
\begin{equation}\label{equ:calculation of the difference of z_{t+1} and z_t}
    \begin{split}
        z_{t+1}-z_t&=\frac{1}{1-\beta_1}(\theta_{t+1}-\theta_t)-\frac{\beta_1}{1-\beta_1}(\theta_t-\theta_{t-1})
        \\&= -\frac{1}{1-\beta_1}A_tm_t+\frac{\beta_1}{1-\beta_1}A_{t-1}m_{t-1}
        \\&=-\frac{1}{1-\beta_1}A_t[\beta_1m_{t-1}+(1-\beta_1)\lambda_t]+\frac{\beta_1}{1-\beta_1}A_{t-1}m_{t-1}
        \\&=\frac{\beta_1}{1-\beta_1}(A_{t-1}-A_t)m_{t-1}-A_t\lambda_t.
    \end{split}
\end{equation}
Combining \cref{equ:calculation of z_2 and z_1} and \cref{equ:calculation of the difference of z_{t+1} and z_t} proves \cref{lem:expression for the difference of z_t}.
\end{proof}

\begin{lemma}\label{lem:upper bound for the norm of the difference of z_t}
      Let $\{z_t\}$ be the auxiliary sequence constructed in \cref{equ:the construction of z_t}, we then have
      \begin{equation*}\label{distance between z_t is bouned by distance between theta_t plus extra term}
          \|z_{t+1}-z_{t}\|_2\leq \frac{\beta_1}{1-\beta_1}\|\theta_t-\theta_{t-1}\|_2+\|A_t\lambda_t\|_2,\quad\forall t\geq 2.
      \end{equation*}
    \end{lemma}

\begin{proof}[Proof]
    By \cref{lem:expression for the difference of z_t}, we obtain
    \begin{equation*}
        z_{t+1}-z_t=\frac{\beta_1}{1-\beta_1}(A_{t-1}-A_t)m_{t-1}-A_t\lambda_t,\quad\forall t\geq 2.
    \end{equation*}
    Since each $A_t$ is invertible for each $t\geq 1$ by construction, the above equation can be equivalently reformulated as
    \begin{equation}\label{equ:the modified version of the difference of z_t}
        z_{t+1}-z_t=\frac{\beta_1}{1-\beta_1}(I-A_tA_{t-1}^{-1})A_{t-1}m_{t-1}-A_t\lambda_t.
    \end{equation}
    Taking Euclidean norm on both sides of \cref{equ:the modified version of the difference of z_t}and invoking the triangle inequality together with the definition of the operator norm, we deduce that
    \begin{equation}\label{equ:the upper bound for the norm of the difference of z_t}
        \begin{split}
            \|z_{t+1}-z_t\|_2\leq \frac{\beta_1}{1-\beta_1}\|I-A_tA_{t-1}^{-1}\|_{2}\|A_{t-1}m_{t-1}\|_2+\|A_{t}\lambda_t\|_2,
        \end{split}
    \end{equation}
    
    Moreover, since $\{A_t\}_{t\geq 1}$ forms a non-decreasing sequence of diagonal matrices, the diagonal entries of $A_{t}A_{t-1}^{-1}$ are bounded above by $1$, which immediately yields the uniform estimate
    \begin{equation}\label{equ:the upper bound for the operator norm}
        \|I-A_{t}A_{t-1}^{-1}\|_{2}\leq 1.
    \end{equation}
    On the other hand, the update rule for $\theta$ ensures that
    \begin{equation}\label{equ:the update formula for parameter theta}
        \theta_{t}-\theta_{t-1}=-A_{t-1}m_{t-1}.
    \end{equation}
    Consequently, combining  \cref{equ:the upper bound for the norm of the difference of z_t}, \cref{equ:the upper bound for the operator norm}, and \cref{equ:the update formula for parameter theta}, we finally get
    \begin{equation*}
        \|z_{t+1}-z_t\|_2\leq \frac{\beta_1}{1-\beta_1}\|\theta_{t}-\theta_{t-1}\|_2+\|A_{t}\lambda_t\|_2.
    \end{equation*}
\end{proof}

    \begin{lemma}\label{lem:upper bound for the norm of difference of gradient at z_t and theta_t}
        Let $\{z_t\}$ be the auxiliary sequence constructed in \cref{equ:the construction of z_t}, we then have
        \begin{equation*}\label{the gradient norm difference evaluated at z_t and theta_t}
            \|\nabla f(z_t)-\nabla f(\theta_t)\|_2\leq L\left(\frac{\beta_1}{1-\beta_1}\right)\|\theta_t-\theta_{t-1}\|_2,\quad\forall t\geq 2.      \end{equation*}
    \end{lemma}

\begin{proof}[Proof]

By the construction \cref{equ:the construction of z_t}, we have
\begin{equation*}
    \|z_t-\theta_t\|_2\leq \frac{\beta_1}{1-\beta_1}\|\theta_{t}-\theta_{t-1}\|_2
\end{equation*}
Since $f$ is $L$-smooth, we thus obtain
\begin{equation*}
    \begin{split}
        \|\nabla f(z_t)-\nabla f(\theta_t)\|_2\leq L\|z_t-\theta_t\|_2\leq L\left(\frac{\beta_1}{1-\beta_1}\right)\|\theta_t-\theta_{t-1}\|,\quad\forall t\geq 2.
    \end{split}
\end{equation*}
\end{proof}

Following the notation in~\cite{zhou-2024-convergence_of_adagrad}, 
for a sequence $\{g_j\}_{j=1}^t$ we denote by $g_{j,i}$ its $i$-th coordinate and set
\begin{equation*}
     g_{1:t,i} = [ g_{1,i}, g_{2,i},\ldots, g_{t,i}]^\top .
\end{equation*}
For the sequence $\{\lambda_j\}_{j=1}^t$ in Algorithm~\ref{alg:svd-adam_algorithm}, 
we use the same convention and write
\begin{equation*}
    \lambda_{1:t,i} = [ \lambda_{1,i}, \lambda_{2,i},\ldots,\lambda_{t,i}]^\top .
\end{equation*}

    \begin{lemma}\label{lem:relation for the cumulative bound of lambda and gradient}
        Let $\{g_t\}_{t\geq 1}$ and $\{\lambda_t\}_{t\geq 1}$ be sequences generated by the modified Algorithm~\ref{alg:svd-adam_algorithm}. Then
        \begin{align}
             &\sum_{i=1}^{m}\|\lambda_{1:T,i}\|_2\leq \sqrt{m}C_2\left(\sum_{i=1}^m\|g_{1:T,i}\|_2^2\right)^{{1}/{2}}\label{equ:upper bound for sum of lambda}
             \\
             &\sum_{i=1}^m\|\lambda_{1:T,i}\|_2^2\leq C_2^2\sum_{i=1}^m\|g_{1:T,i}\|_2^2
             \label{equ:upper bound for quadratic sum of lambda}
        \end{align}
    \end{lemma}

\begin{proof}[Proof]

By the Cauchy-Schwarz inequality,
\begin{equation*}
    \sum_{i=1}^m \|\lambda_{1:T,i}\|_2
\leq \sqrt{m} \left( \sum_{i=1}^m \|\lambda_{1:T,i}\|_2^2 \right)^{\frac{1}{2}}
= \sqrt{m} \left( \sum_{t=1}^T \|\lambda_t\|_2^2 \right)^{\frac{1}{2}}.
\end{equation*}
Since $\lambda_t = B_t g_t$ and $\|B_t\|_2 \le C_2$, we have
\begin{equation*}
    \|\lambda_t\|_2 \le \|B_t\|_2 \|g_t\|_2 \le C_2 \|g_t\|_2.
\end{equation*}
Substituting this bound gives
\[
\sum_{i=1}^m \|\lambda_{1:T,i}\|_2
\le \sqrt{m}\, C_2 \left( \sum_{t=1}^T \|g_t\|_2^2 \right)^{1/2}
= \sqrt{m}\, C_2 \left( \sum_{i=1}^m \|g_{1:T,i}\|_2^2 \right)^{1/2},
\]
which proves~\cref{equ:upper bound for sum of lambda}. 
For~\cref{equ:upper bound for quadratic sum of lambda}, note that
\begin{equation*}
    \begin{split}
        \sum_{i=1}^m \|\lambda_{1:T,i}\|_2^2
&= \sum_{t=1}^T \|\lambda_t\|_2^2
\leq \sum_{t=1}^T \|B_t\|_2^2 \|g_t\|_2^2
\\
&\leq C_2^2 \sum_{t=1}^T \|g_t\|_2^2
= C_2^2 \sum_{i=1}^m \|g_{1:T,i}\|_2^2.
    \end{split}
\end{equation*}
This completes the proof.
\end{proof}

The following two lemmas, adapted from Lemma~6.5 in~\cite{zhou-2024-convergence_of_adagrad}, provide bounds on $\{A_tm_t\}_{t\ge1}$ and $\{A_t\lambda_t\}_{t\geq 1}$. 
More specifically, \cref{lem:the upper bound for the Quadratic Variation of theta_t} establishes quadratic-variation bounds for $\{A_t m_t\}_{t\geq 1}$ and $\{A_t\lambda_t\}_{t\geq 1}$, which are controlled by $\sum_{i=1}^m \|\lambda_{1:T,i}\|_2$.  
In contrast, \cref{lem:the upper bound for the variation of theta_t} provides bounds on their total variation, which depend on $\left(\sum_{i=1}^m \|\lambda_{1:T,i}\|_2 \right)^{1/2}$.

    \begin{lemma}\label{lem:the upper bound for the Quadratic Variation of theta_t}
        Assume that $\beta_1<\sqrt{\beta_2}$. If $\{\theta_t\}_{t\geq 1}$ are parameters generated by the modified Algorithm~\ref{alg:svd-adam_algorithm}, then
        \begin{align}   
        \sum_{t=1}^T\|A_tm_t\|^2_2&\leq \frac{\alpha^2(1-\beta_1)}{2\epsilon^{\frac{1}{2}}(1-\beta_2)^{\frac{1}{2}}(1-\gamma)}T^{\frac{1}{2}}\sum_{i=1}^m\|\lambda_{1: T,i}\|_2 \label{variation of Atmt} 
        \\&=(1-\beta_1)\alpha^2KT^{\frac{1}{2}}\sum_{i=1}^m\|\lambda_{1: T,i}\|_2,\notag
        \\ 
        \sum_{t=1}^T\|A_t\lambda_t\|_2^2&\leq \frac{\alpha^2}{2\epsilon^{\frac{1}{2}}(1-\beta_2)^{\frac{1}{2}}(1-\gamma)}T^{\frac{1}{2}}\sum_{i=1}^m\|\lambda_{1: T,i}\|_2
        \\&=\alpha^2KT^{\frac{1}{2}}\sum_{i=1}^m\|\lambda_{1: T,i}\|_2,\notag 
        \label{variation of Atlambdat}
        \end{align}
        where the constant $K$ is defined as
        \begin{equation*}
            K=\frac{1}{2\epsilon^{\frac{1}{2}}(1-\beta_2)^{\frac{1}{2}}(1-\gamma)}.
        \end{equation*}
    \end{lemma}

\begin{proof}[Proof]
Let $\widehat v_{t,i}$, $m_{t,i}$, and $\lambda_{t,i}$ denote the $i$-th coordinates of
$\widehat v_t$, $m_t$, and $\lambda_t$, respectively. 
Since $A_t = \alpha_t \widehat D_t^{-1/2}$, we have
\begin{equation}
\label{equ:At_mt_def}
\|A_t m_t\|_2^2
= \alpha_t^2 \sum_{i=1}^{m} \frac{m_{t,i}^2}{\widehat v_{t,i} + \epsilon}
\le
\frac{\alpha_t^2}{2\,\epsilon^{1/2}}
\sum_{i=1}^{m} \frac{m_{t,i}^2}{v_{t,i}^{1/2}},
\end{equation}
where we used $a+b\ge 2\sqrt{ab}$ and $\widehat v_{t,i}\ge v_{t,i}$. Expanding $m_t$ and $v_t$ recursively,
\[
m_t = (1-\beta_1)\sum_{j=1}^{t}\beta_1^{t-j}\lambda_j,
\qquad
v_t = (1-\beta_2)\sum_{j=1}^{t}\beta_2^{t-j}\lambda_j^2.
\]
Taking the $i$-th coordinate and substituting into \cref{equ:At_mt_def} gives 
\begin{equation} \label{equ:At_mt_expand} 
\|A_t m_t\|_2^2 \leq \frac{\alpha_t^2(1-\beta_1)^2}{2\epsilon^{1/2}(1-\beta_2)^{1/2}} \sum_{i=1}^{m} \left(\sum_{j=1}^{t}\beta_1^{t-j}\lambda_{j,i}\right)^{2} \left(\sum_{j=1}^{t}\beta_2^{t-j}\lambda_{j,i}^2\right)^{-1/2}.
\end{equation}
Applying the Cauchy-Schwarz inequality to \cref{equ:At_mt_expand}
\begin{equation}\label{equ:At_mt_cauchy}
\left(\sum_{j=1}^{t}\beta_1^{t-j}\lambda_{j,i}\right)^{2}
\le
\left(\sum_{j=1}^{t}\beta_1^{t-j}\right)
\left(\sum_{j=1}^{t}\beta_1^{t-j}\lambda_{j,i}^{2}\right)
\le\frac{1}{1-\beta_1}
\sum_{j=1}^{t}\beta_1^{t-j}\lambda_{j,i}^2.
\end{equation}
Substituting \cref{equ:At_mt_cauchy} into \cref{equ:At_mt_expand} yields
\begin{equation}
\label{equ:At_mt_cauchy_sub_fixed}
\|A_t m_t\|_2^2
\le
\frac{\alpha_t^2(1-\beta_1)}{2\,\epsilon^{1/2}(1-\beta_2)^{1/2}}
\sum_{i=1}^{m}\frac{\sum_{j=1}^t\beta_1^{t-j}\lambda_{j,i}^{2}}
{\big(\sum_{k=1}^{t}\beta_2^{\,t-k}\lambda_{k,i}^2\big)^{1/2}}.
\end{equation}
For each $i$ and $j\le t$,
\[
\Big(\sum_{k=1}^{t}\beta_2^{\,t-k}\lambda_{k,i}^2\Big)^{1/2}
\ge \big(\beta_2^{\,t-j}\lambda_{j,i}^2\big)^{1/2}
=\beta_2^{\frac{t-j}{2}}|\lambda_{j,i}|,
\]
hence with $\gamma:=\beta_1/\beta_2^{1/2}$,
\begin{equation}
\label{equ:ratio_gamma_step}
\frac{\beta_1^{t-j}\lambda_{j,i}^{2}}
{\Big(\sum_{k=1}^{t}\beta_2^{\,t-k}\lambda_{k,i}^2\Big)^{1/2}}
\le \gamma^{\,t-j}\,|\lambda_{j,i}|.
\end{equation}
Combining \eqref{equ:At_mt_cauchy_sub_fixed} and \eqref{equ:ratio_gamma_step},
\begin{equation}
\label{equ:At_mt_gamma_sum_pre}
\|A_t m_t\|_2^2
\le
\frac{\alpha_t^2(1-\beta_1)}{2\,\epsilon^{1/2}(1-\beta_2)^{1/2}}
\sum_{i=1}^{m}\sum_{j=1}^{t}\gamma^{\,t-j}|\lambda_{j,i}|.
\end{equation}
Summing \eqref{equ:At_mt_gamma_sum_pre} over $t=1,\dots,T$ and noticing that $\gamma<1$,
\begin{equation}
\label{equ:At_mt_sum_general_alpha}
\sum_{t=1}^{T}\|A_t m_t\|_2^2
\leq
\frac{(1-\beta_1)}{2\epsilon^{1/2}(1-\beta_2)^{1/2}}
\sum_{i=1}^{m}\sum_{j=1}^{T}\Big(\sum_{t=j}^{T}\alpha_t^2\,\gamma^{t-j}\Big)|\lambda_{j,i}|.
\end{equation}
In particular, since $\alpha_t\equiv\alpha$, $\sum_{t=j}^{T}\alpha^2\gamma^{\,t-j}\le \alpha^2/(1-\gamma)$, so
\begin{equation}
\label{equ:At_mt_gamma_final}
\sum_{t=1}^{T}\|A_t m_t\|_2^2
\leq
\frac{\alpha^2(1-\beta_1)}{2\,\epsilon^{1/2}(1-\beta_2)^{1/2}(1-\gamma)}
\sum_{i=1}^{m}\sum_{j=1}^{T}|\lambda_{j,i}|.
\end{equation}
Finally, apply the Cauchy-Schwarz inequality again,
\begin{equation*}
\sum_{i=1}^{m}\sum_{j=1}^{T}|\lambda_{j,i}|
\le
T^{1/2}\sum_{i=1}^{m}\|\lambda_{1:T,i}\|_2.
\end{equation*}
Thus,
\begin{equation*}
\sum_{t=1}^{T}\|A_t m_t\|_2^2
\le
\frac{T^{1/2}\alpha^2(1-\beta_1)}{2\,\epsilon^{1/2}(1-\beta_2)^{1/2}(1-\gamma)}
\sum_{i=1}^{m}\|\lambda_{1:T,i}\|_2.
\end{equation*}
This completes the proof of \cref{equ:variation for Atmt}. For the proof of \cref{equ:variation for Atlambdat}, we obtain
\begin{equation*}
\sum_{t=1}^{T}\|A_t\lambda_t\|_2^2
\le
\frac{T^{1/2}\alpha^2}{2\,\epsilon^{1/2}(1-\beta_2)^{1/2}}
\sum_{i=1}^{m}\|\lambda_{1:T,i}\|_2.
\end{equation*}

\end{proof}

    \begin{lemma}\label{lem:the upper bound for the variation of theta_t}
        Assume that $\beta_1<\sqrt{\beta_2}$. If $\{\theta_t\}_{t\geq 1}$ are parameters generated by the modified Algorithm~\ref{alg:svd-adam_algorithm}, then
        \begin{align}   
        \sum_{t=1}^T\|A_tm_t\|_2&\leq \frac{\alpha(1-\beta_1)^{\frac{1}{2}}}{2^{\frac{1}{2}}\epsilon^{\frac{1}{4}}(1-\beta_2)^{\frac{1}{4}}(1-\gamma)^{\frac{1}{2}}}T^{\frac{3}{4}}\left(\sum_{i=1}^m\|\lambda_{1: T,i}\|_2\right)^{\frac{1}{2}}\label{equ:variation for Atmt} 
        \\
        &=\alpha[(1-\beta_1)K]^{\frac{1}{2}}T^{\frac{3}{4}}\left(\sum_{i=1}^m\|\lambda_{1: T,i}\|_2\right)^{\frac{1}{2}},\notag
        \\ 
        \sum_{t=1}^T\|A_t\lambda_t\|_2&\leq \frac{\alpha}{2^{\frac{1}{2}}\epsilon^{\frac{1}{4}}(1-\beta_2)^{\frac{1}{4}}(1-\gamma)^{\frac{1}{2}}}T^{\frac{3}{4}}\left(\sum_{i=1}^m\|\lambda_{1: T,i}\|_2\right)^{\frac{1}{2}},
        \label{equ:variation for Atlambdat}
        \\
        &=\alpha K^{\frac{1}{2}}T^{\frac{3}{4}}\left(\sum_{i=1}^m\|\lambda_{1: T,i}\|_2\right)^{\frac{1}{2}}.\notag
        \end{align}
    \end{lemma}

\begin{proof}[Proof]
    By \cref{lem:the upper bound for the Quadratic Variation of theta_t}, we have already obtained
    \begin{equation}\label{equ:qudratic variation in proof of lemma}
        \sum_{t=1}^T\|A_tm_t\|^2_2\leq \frac{\alpha_t^2(1-\beta_1)}{2\epsilon^{\frac{1}{2}}(1-\beta_2)^{\frac{1}{2}}(1-\gamma)}T^{\frac{1}{2}}\sum_{i=1}^m\|\lambda_{1: T,i}\|_2.
    \end{equation}
Applying the Cauchy-Schwartz inequality to \cref{equ:qudratic variation in proof of lemma}, we get
\begin{equation*}
    \left(\sum_{t=1}^T\|A_tm_t\|_2\right)^2\leq T\sum_{t=1}^T\|A_tm_t\|_2^2.
\end{equation*}
This implies \cref{equ:variation for Atmt} immediately. \cref{equ:variation for Atlambdat} can be proven in the same way.
\end{proof}

\bibliographystyle{abbrv}
\bibliography{references}